%% file: main.tex
\documentclass[11pt]{amsart}
\usepackage[utf8]{inputenc}
\usepackage{amssymb,amsmath}
\usepackage{graphicx}
\usepackage{color}
\usepackage{tikz}
\usepackage{pgfplots}
\usepackage{enumitem}
\usepackage{mathtools}
\usepackage[hidelinks]{hyperref}
\usepackage{xfrac}

\usepackage{accents}
\usepackage{multicol} 
\usepackage{soul}
\usepackage{subcaption}
\newtheorem{theorem}{Theorem}
\newtheorem{proposition}[theorem]{Proposition}

\theoremstyle{definition}

\theoremstyle{remark}
\newtheorem{remark}[theorem]{Remark}

\numberwithin{theorem}{section}
\numberwithin{equation}{section}
\numberwithin{table}{section}
\numberwithin{figure}{section}
\input{def}
\allowdisplaybreaks
\begin{document}
\title[Dissipation-preserving discretization of the Cahn--Hilliard equation]{Dissipation-preserving discretization of the\\ Cahn--Hilliard equation with dynamic\\ boundary conditions}
\author[]{R.~Altmann$^\dagger$, C.~Zimmer$^{\dagger}$}
\address{${}^{\dagger}$ Department of Mathematics, University of Augsburg, Universit\"atsstr.~14, 86159 Augsburg, Germany}
\email{robert.altmann@math.uni-augsburg.de, christoph.zimmer@math.uni-augsburg.de}
\thanks{The authors acknowledge the support of the Deutsche Forschungsgemeinschaft (DFG, German Research Foundation) through the project 446856041.}
\date{\today}
\begin{abstract}
This paper deals with time stepping schemes for the Cahn--Hilliard equation with three different types of dynamic boundary conditions. The proposed schemes of first and second order are mass-conservative and energy-dissipative and -- as they are based on a formulation as a coupled system of partial differential equations -- allow different spatial discretizations in the bulk and on the boundary. The latter enables refinements on the boundary without an adaptation of the mesh in the interior of the domain. The resulting computational gain is illustrated in numerical experiments.
\end{abstract}
%
%
\maketitle
%
%
%
%
\section{Introduction}
Renowned mathematical models describing the phase separation of binary mixtures include the Allen--Cahn as well as the here considered Cahn--Hilliard equation~\cite{CahH58,EllS86}. Originally proposed in the field of material science, nowadays, the Cahn--Hilliard equation is successfully applied in several physical areas, e.g., to model electrokinetic phenomena by a coupling with the Navier--Stokes equations~\cite{CamGK12}.

As we deal with partial differential equations, boundary conditions are needed to complete the system. The simplest model considers (homogeneous) Neumann boundary conditions for the phase-field variable as well as for the chemical potential.  
In recent years, however, rising attention has been attracted by a new class of boundary conditions that properly reflect effective properties on the surface of the domain. These so-called {\em dynamic boundary conditions} are itself a differential equation that incorporate an energy on the surface. In~\cite{KenEMetal01} it is proposed to use an Allen--Cahn equation on the boundary, whereas a model for non-permeable walls was suggested in~\cite{GolMS11}. More recently, a new model accounting for possible short-range interactions of the material with the solid wall was introduced in~\cite{LiuW19} and further analyzed in~\cite{GarK20}. A combination of the latter two models is considered in~\cite{KnoLLM21}.

Concerning the numerical treatment of the Cahn--Hilliard equation, we focus in this paper on the temporal discretization. For results on the spatial discretization, we refer to \cite{ChePP10,CheP14,HarK21,KnoLLM21} and the references therein. 
The use of a convex--concave splitting of the nonlinearity in the context of the Cahn--Hilliard equation was already proposed in~\cite{Eyr98}. Later, it was further applied to different types of boundary conditions; see, e.g., \cite{Gru13,GuaWW14}. 
Yet another approach is to treat the nonlinearity {\em explicitly}. This, however, often comes at the price of an additional stabilization parameter, which depends on the solution itself~\cite{BaoZ21}. Hence, a large stabilization parameter is necessary in theory, which in turn yields inaccurate numerical results. For recent findings on the stability in combination with \emph{large} time steps, we refer to~\cite{Li21}. A convergence analysis of such an implicit--explicit scheme without a stabilization term (in combination with standard boundary conditions) is given in~\cite{LiQT22}. Here, however, a time step restriction is necessary. 
Finally, in the context of systems with dynamic boundary conditions (but not in connection with Cahn--Hilliard) we would like to mention existing splitting approaches, which use similar ideas as in this paper. Therein, the aim is to decouple bulk and surface dynamics, leading to more efficient time stepping schemes~\cite{KovL17,AltV21,AltKZ22}. \smallskip 

Within this paper, we consider the Cahn--Hilliard equation with three different types of dynamic boundary conditions. For all three models, we present in Section~\ref{sec:CHmodel} an alternative weak formulation as a {\em partial differential-algebraic equation} (PDAE); see~\cite{LamMT13,Alt15} for an introduction. These formulations are characterized by the fact that additional variables are introduced on the boundary. As a result, we consider the Cahn--Hilliard equation in the bulk and the boundary conditions as two systems which are coupled through certain constraints acting only on the boundary. This then enables more flexibility for the spatial as well as the temporal discretization. At the same time, the models maintain the crucial properties of mass-conservation and energy-dissipation. 

Section~\ref{sec:firstOrder} is devoted to first-order time stepping schemes which are dissipation-preserving. Here, we follow the already mentioned strategy of a convex--concave splitting of the nonlinearity. Treating the convex part implicitly and the concave part explicitly, we show for all models that the property of being energy-dissipative is maintained after discretization. Moreover, since the schemes are based on a PDAE formulation, an additional refinement on the boundary -- in time and space -- is possible without refining the considered mesh in the bulk. This flexibility is of particular value if the solution oscillates rapidly on the boundary as we illustrate in the numerical experiments of Section~\ref{sec:firstOrder:numerics}. 
Second-order time stepping schemes of Crank--Nicolson type are then discussed in Section~\ref{sec:secondOrder}. Here, we restrict ourselves to the classical double-well potential. Again, we prove that the discretization maintains the property of being energy-dissipative, still allowing the flexible use of spatial discretization schemes. 
Finally, we conclude in Section~\ref{sec:conclusion}. 
%
%
\section{Abstract Formulation of the Cahn--Hilliard Equation} 
\label{sec:CHmodel} 
The Cahn--Hilliard equation in a bounded spatial domain $\Omega \subset \R^d$, $d\in\{2, 3\}$, with time horizon $T<\infty$ is given by 
\begin{subequations}
\label{eqn:CahnHilliard_bulk}
\begin{alignat}{3}
	\dot u - \sigma\Delta w 
	&= 0  &&\qquad\text{in } \Omega \times [0,T],\label{eqn:CahnHilliard_bulk_a} \\ 
	- \eps\, \Delta u + \eps^{-1} W^\prime(u) 
	&= w &&\qquad \text{in } \Omega \times [0,T] \label{eqn:CahnHilliard_bulk_b}
\end{alignat}
\end{subequations}
with an initial condition~$u(0) = u^0$. 
Therein, $u$ equals the \emph{phase-field parameter} with codomain $[-1,1]$ and represents the local relative concentration for the two components. The values~$u = \pm 1$ correspond to a pure component, whereas values in $(-1,1)$ correspond to the transition of the mixture. The separation is driven by the \emph{chemical potential}~$w$. The constant $\eps>0$ is the so-called \emph{interaction length} and describes the thickness of the transition area of one component to another, whereas $\sigma>0$ is a dissipation parameter. Finally, $W$ denotes the \emph{potential} with typical examples including the polynomial or logarithmic double-well potential.  

For the completion of system~\eqref{eqn:CahnHilliard_bulk}, we need boundary conditions for $u$ and $w$ on $\Gamma\coloneqq \partial \Omega$. Besides standard Neumann boundary conditions, we consider three different types of dynamic boundary conditions in the sequel. For all four cases, we present an abstract operator formulation and discuss the conservation of mass and the dissipation of energy. In Section~\ref{sec:CHmodel:illustration}, we provide an illustrative comparison of the four models. 
%
%
\subsection{Homogeneous Neumann boundary conditions}\label{sec:CHmodel:Neumann} 
A classical choice for the completion of the Cahn--Hilliard system~\eqref{eqn:CahnHilliard_bulk} are homogeneous Neumann conditions, i.e.,
\begin{subequations}
\label{eqn:CahnHilliard_boundary_Neumann}
\begin{alignat}{3}
	\partial_n w &= 0 &&\qquad \text{on } \Gamma \times [0,T], \label{eqn:CahnHilliard_boundary_Neumann_a}\\ 
	\partial_n u &= 0 &&\qquad \text{on } \Gamma \times [0,T].  \label{eqn:CahnHilliard_boundary_Neumann_b}	
\end{alignat}
\end{subequations}
This corresponds to the physical interpretation that the material and the surrounding wall do not interact. 
%
These homogeneous boundary conditions directly imply the conservation of mass. To see this, we integrate by parts and apply~\eqref{eqn:CahnHilliard_boundary_Neumann_a}, leading to 
\[
  \ddt \int_\Omega u \dx 
  = \int_\Omega \dot u \dx 
  = \sigma \int_\Omega \Delta w \dx 
  = \sigma \int_{\Gamma} \partial_n w \dx 
  = 0. 
\]
%
On the other hand, the \emph{bulk free energy} is defined by
\begin{equation}
\label{eqn:CahnHilliard_energy_bulk}
  E_{\text{bulk}}(u)
  \coloneqq \int_\Omega \frac \eps 2\, \nabla u \cdot \nabla u + \frac 1 \eps\, W(u) \dx
\end{equation}
and is dissipative. More precisely, we have 
\[
  \ddt E_{\text{bulk}}(u) 
  = \int_\Omega \eps\, \nabla u \cdot \nabla \dot u + \eps^{-1} W^\prime(u)\, \dot u \dx 
  \stackrel{\eqref{eqn:CahnHilliard_bulk_b}}{=} \int_\Omega w\, \dot u \dx \stackrel{\eqref{eqn:CahnHilliard_bulk_a}}{=} - \sigma \int_\Omega \nabla w \cdot \nabla w \dx \leq 0.
\]

Finally, we would like to introduce an abstract operator formulation, which corresponds to the weak formulation of the system. For this, we introduce the trial space $\calV \coloneqq H^1(\Omega)$ for both variables $u$ and $w$. 
\begin{remark}
\label{rem:VuVw}
For the spatial discretization, it may be of interest to distinguish the trial spaces for $u$ and $w$ in order to allow different discretization schemes. This, however, is not the focus of this work. 
\end{remark}
Moreover, we introduce the differential operator~$\calK_\Omega\colon \calVu\to \calVu^\ast$ by 
\[
  \langle \calK_\Omega u, v \rangle 
  \coloneqq \int_\Omega \nabla u \cdot \nabla v \dx.
\]
Then, integrating by parts and applying the homogeneous Neumann boundary conditions, system~\eqref{eqn:CahnHilliard_bulk} can be written as the PDAE 
\begin{subequations}
\label{eqn:opEqn:Neumann}
\begin{alignat}{3}
  \dot u + \sigma \calK_\Omega w 
  &= 0 &&\qquad \text{in } \calVw^\ast,\\
  \eps\, \calK_\Omega u + \eps^{-1} W^\prime(u) 
  &= w &&\qquad \text{in } \calVu^\ast.
\end{alignat}
\end{subequations}
Note that both equations are stated in the dual space of $\calV$ (indicating the space of test functions) and that they should hold for a.e.~$t\in[0,T]$. Further note that system~\eqref{eqn:opEqn:Neumann} is indeed a PDAE as only the first equation contains a time derivative and the second equation yields an algebraic equation after spatial discretization. Moreover, we would like to emphasize that the conservation of mass and the dissipation of energy also hold for this weak formulation. Here, these two properties read 
\[
  \ddt (u, 1)_{L^2(\Omega)}
  = - \sigma \langle \calK_\Omega w, 1\rangle
  = 0, \qquad
  \ddt E_{\text{bulk}}(u) 
  = (w, \dot u)_{L^2(\Omega)}
  = - \langle \calK_\Omega w, w \rangle 
  \le 0. 
\]

For many applications, the non-interaction of the material with the wall due to the homogeneous Neumann boundary conditions~\eqref{eqn:CahnHilliard_boundary_Neumann} is rather restrictive. In order to describe short-range interactions between the solid wall and the mixture, physicists introduced a suitable surface free energy functional and more complex boundary conditions. 
%
%
\subsection{Allen--Cahn type boundary conditions}\label{sec:CHmodel:AllenCahn} 
As a first example of more complex boundary conditions, we consider the model of Allen--Cahn type suggested by Kenzler et al., cf.~\cite{KenEMetal01}, given by
\begin{subequations}
	\label{eqn:CahnHilliard_boundary_AC}
	\begin{alignat}{3}
	\partial_n w 
	&= 0 &&\qquad \text{on } \Gamma \times [0,T], \label{eqn:CahnHilliard_boundary_AC_a}\\
	\dot u - \delta \kappa\, \Delta_\Gamma u + \delta^{-1} W^\prime_\Gamma(u) + \eps\, \partial_n u 
	&= 0 &&\qquad \text{on } \Gamma \times [0,T]. \label{eqn:CahnHilliard_boundary_AC_b}	
	\end{alignat}
\end{subequations}
Note that the chemical potential $w$ still does not interact with the solid wall. Nevertheless, the binary mixture separates on the boundary, where the separation is described by the Allen--Cahn equation acting on the manifold given by the boundary. The parameter~$\delta$ denotes the interaction length on the boundary (similar to $\eps$ in $\Omega$), $W_\Gamma$ equals the \emph{boundary energy potential}, and $\Delta_\Gamma$ is the \emph{Laplace--Beltrami operator}~\cite[Ch.~16.1]{GilT01} with corresponding dissipation parameter~$\kappa>0$. 

Since the boundary condition for $w$ is the same as in the Neumann case in~\eqref{eqn:CahnHilliard_boundary_Neumann_a}, we can derive the same calculation to prove conservation of mass. Note, however, that the mass on~$\Gamma$ is not conserved, in general. This can also be observed in numerical simulations.
%
We now turn to the energy of the system. Besides the bulk free energy introduced in~\eqref{eqn:CahnHilliard_energy_bulk}, the model also involves a \emph{surface free energy}, namely 
\[
  E_{\text{surf}}(u) 
  \coloneqq \int_\Gamma \frac {\delta \kappa} 2\, \nabla_\Gamma u \cdot \nabla_\Gamma u + \frac 1 \delta\, W_\Gamma(u) \dx.
\]
Hence, the total energy is given by $E(u) \coloneqq E_{\text{bulk}}(u) + E_{\text{surf}}(u)$, 
%
%
which is again dissipative. To see this, a similar calculation as in the previous section shows with~\eqref{eqn:CahnHilliard_bulk} and~\eqref{eqn:CahnHilliard_boundary_AC_b} that 
\begin{equation*}
  \ddt E(u) 
  = - \sigma \int_\Omega \nabla w \cdot \nabla w \dx - \int_\Gamma \dot{u}^2 \dx 
  \leq 0.
\end{equation*}

To obtain an abstract operator formulation, which is suitable for numerical simulations and which enables a separate treatment of the boundary dynamics, we follow the procedure of~\cite{Alt19,AltKZ22}. For this, we introduce an auxiliary variable $p\coloneqq u|_\Gamma$ and the trace spaces
\[
  \calVp \coloneqq H^1(\Gamma), \qquad 
  \calQ \coloneqq H^{-1/2}(\Gamma).
\]
With the typical trace operator~$\calB\colon \calVu = H^1(\Omega)\to H^{1/2}(\Gamma) = \calQ^\ast$, the connection of~$u$ and~$p$ can be expressed in the form $\calB u - p = 0$ as equation in~$\calQ^\ast$. We regard this equation as a constraint, which we include by the help of a Lagrange multiplier~$\lambda$. 
Moreover, we introduce the differential operator~$\calK_\Gamma\colon\calVp\to\calVp^\ast$ by
\[
  \langle \calK_\Gamma p, q \rangle 
  \coloneqq \int_\Gamma \nabla_\Gamma p \cdot \nabla_\Gamma q \dx.
\]

The resulting operator formulation of~\eqref{eqn:CahnHilliard_bulk} with the Allen--Cahn type boundary conditions~\eqref{eqn:CahnHilliard_boundary_AC} then leads to the following PDAE: 
seek $u, w\colon [0,T] \to \calV$, $p\colon [0,T] \to \calVp$, and $\lambda\colon [0,T] \to \calQ$ such that 
\begin{subequations}
\label{eqn:opEqn:AllenCahn}
\begin{alignat}{3}
  \dot u + \sigma \calK_\Omega w &= 0 &&\qquad \text{in } \calVw^\ast,\\
  \eps\, \calK_\Omega u + \eps^{-1} W^\prime(u) - \eps\, \calB^\ast \lambda &= w &&\qquad \text{in } \calVu^\ast, \\
  \dot p + \delta \kappa\, \calK_\Gamma p + \delta^{-1} W^\prime_\Gamma(p) + \eps\,  \lambda &= 0 &&\qquad \text{in } \calVp^\ast, \\
  \calB u - p  &= 0 &&\qquad \text{in } \calQ^\ast \label{eqn:opEqn:AllenCahn:d}
\end{alignat}
\end{subequations}
for a.e.~$t\in[0,T]$. As initial data, we expect given~$u(0)=u^0$ and $p(0)=p^0$, which we call \emph{consistent} if $u^0|_\Gamma = p^0$, i.e., if the initial data satisfies the constraint~\eqref{eqn:opEqn:AllenCahn:d}. 
Again, the introduced conservation and dissipation properties can be shown for this weak formulation, leading to  
\[
  \ddt E(u)
  = (w, \dot u)_{L^2(\Omega)} - (\dot p, \dot p)_{L^2(\Gamma)}
  + \eps \langle \calB \dot u - \dot p, \lambda\rangle
  = - \sigma\langle \calK_\Omega w, w\rangle - (\dot p, \dot p)_{L^2(\Gamma)} 
  \le 0.
\]
\begin{remark}
A spatial discretization of system~\eqref{eqn:opEqn:AllenCahn} (as well as the following systems~\eqref{eqn:opEqn:LiuWu} and~\eqref{eqn:opEqn:Goldstein}) yields a differential-algebraic equation of index~$2$, cf.~\cite{HaiW96}. In general, such a system of {\em higher index} leads to numerical instabilities in terms of the temporal discretization. Here, however, the constraint~\eqref{eqn:opEqn:AllenCahn:d} has a homogeneous right-hand side such that its derivatives can be computed in an exact manner. As a result, no numerical difficulties occur in the presence of consistent initial data; cf.~\cite[p.~33]{HaiLR89}. 
\end{remark}
In the following two subsections, we consider boundary conditions of Cahn--Hilliard type. In order to distinguish the two models, we address them by the names of the authors, who original introduced them. 
%
%
\subsection{Boundary conditions of Liu and Wu}\label{sec:CHmodel:LiuWu} 
As a first example of boundary conditions of Cahn--Hilliard type, we consider the model derived by Liu and Wu; see~\cite{LiuW19}. Their construction is driven by physical properties, namely conservation of mass, dissipation of energy, and force balance and reads 
\begin{subequations}
\label{eqn:CahnHilliard_boundary_LW}
\begin{alignat}{3}
	\partial_n w 
	&= 0 &&\qquad \text{on } \Gamma \times [0,T], \label{eqn:CahnHilliard_boundary_LW_a}\\
	\dot u - \Delta_\Gamma w_{\Gamma} 
	&= 0 &&\qquad \text{on } \Gamma \times [0,T], \label{eqn:CahnHilliard_boundary_LW_b}\\ 
	- \delta \kappa\, \Delta_\Gamma u + \delta^{-1} W^\prime_\Gamma(u) + \eps\, \partial_n u 
	&= w_{\Gamma} &&\qquad \text{on } \Gamma \times [0,T]. \label{eqn:CahnHilliard_boundary_LW_c}
\end{alignat}
\end{subequations}
Here, $u$ and $w$ are the traces of the associated bulk-states. The \emph{boundary chemical potential} $w_{\Gamma}$, however, is a new independent state, i.e., we do not assume that $w_\Gamma$ equals $w$ on the boundary.  

With the same arguments as before, one shows that the mass in the bulk is constant. 
Moreover, using that $\Gamma$ is 'periodic', we have in addition 
\begin{equation*}
  \ddt \int_\Gamma u \dx  
  = \int_\Gamma \dot u \dx 
  \stackrel{\eqref{eqn:CahnHilliard_boundary_LW_b}}{=} \int_\Gamma \Delta_\Gamma w \dx 
  =0.
\end{equation*} 
Hence, the given model conserves the mass in the bulk as well as on the boundary. 
%
The energy corresponding to~\eqref{eqn:CahnHilliard_bulk} with boundary conditions~\eqref{eqn:CahnHilliard_boundary_LW} is again given by $E(u) = E_{\text{bulk}}(u) + E_{\text{surf}}(u)$. Similar to the calculation in Section~\ref{sec:CHmodel:Neumann} for the Neumann case, one shows energy-dissipation of the form 
\begin{equation*}
  \ddt E(u) 
  = - \sigma \int_\Omega \nabla w \cdot \nabla w \dx -\int_\Gamma \nabla_\Gamma w_\Gamma \cdot \nabla_\Gamma w_\Gamma \dx 
  \leq 0.
\end{equation*}

For the abstract formulation, we follow the procedure of the previous section and introduce the auxiliary variable~$p\coloneqq u|_\Gamma$. 
Here, we consider the same trial space for $p$ and $w_\Gamma$. Note that this may be generalized in order to allow different spatial discretizations, cf.~Remark~\ref{rem:VuVw}. This then leads to the following abstract formulation:  
seek $u, w\colon [0,T] \to \calV$, $p, w_\Gamma\colon [0,T] \to \calVp$, and $\lambda\colon [0,T] \to \calQ$ such that 
\begin{subequations}
\label{eqn:opEqn:LiuWu}
\begin{alignat}{3}
  \dot u + \sigma \calK_\Omega w 
  &= 0 &&\qquad \text{in } \calVw^\ast,\\
  \eps\, \calK_\Omega u + \eps^{-1} W^\prime(u) - \eps\, \calB^\ast \lambda 
  &= w &&\qquad \text{in } \calVu^\ast, \\
  \dot p + \calK_\Gamma w_\Gamma 
  &= 0 &&\qquad \text{in } \calVp^\ast, \\
  \delta\kappa\, \calK_\Gamma p + \delta^{-1} W^\prime_\Gamma(p) + \eps\, \lambda 
  &= w_\Gamma &&\qquad \text{in } \calVp^\ast, \\
  \calB u - p 
  &= 0 &&\qquad \text{in } \calQ^\ast
\end{alignat}
\end{subequations}
for a.e.~$t\in[0,T]$ and prescribed initial data $u(0)=u^0$ and $p(0)=p^0$. Again, the mentioned conservation and dissipation properties are maintained for the weak formulation~\eqref{eqn:opEqn:LiuWu}. 
%
%
%
%
\subsection{Boundary conditions of Goldstein, Miranville, and Schimpera}\label{sec:CHmodel:Goldstein} 
The dynamic boundary conditions introduced in~\cite{GolMS11} model non-permeable walls and read 
\begin{subequations}
\label{eqn:CahnHilliard_boundary_GMS}
\begin{alignat}{3}
	\dot u - \Delta_\Gamma w + \sigma \partial_n w 
	&= 0 &&\qquad \text{on } \Gamma \times [0,T], \label{eqn:CahnHilliard_boundary_GMS_a}\\ 
	- \delta \kappa\, \Delta_\Gamma u + \delta^{-1} W^\prime_\Gamma(u) + \eps\, \partial_n u 
	&= w &&\qquad \text{on } \Gamma \times [0,T]. \label{eqn:CahnHilliard_boundary_GMS_b}	
\end{alignat}
\end{subequations}
Note that, here, $u$ and $w$ denote the variables from the bulk restricted to the boundary. As in all previous examples, we discuss the change of mass and energy over time. 

For the mass, we obtain
\begin{equation*}
  \ddt \int_\Omega u \dx + \ddt \int_\Gamma u \dx 
  = \sigma \int_\Omega \Delta w \dx + \int_\Gamma \Delta_\Gamma w  - \sigma\,\partial_n w \dx 
  = \sigma\int_\Gamma \partial_n w - \partial_n w \dx 
  = 0.
\end{equation*}
This means that the sum of the masses in $\Omega$ and on $\Gamma$ is conserved. In contrast to the model of Section~\ref{sec:CHmodel:LiuWu}, however, the single masses may change over time, which can also be observed numerically. 
%
The total energy, which again consists of the bulk and surface free energy, satisfies 
\begin{equation*}
  \ddt E(u) 
  = - \sigma\int_\Omega \nabla w \cdot \nabla w \dx -\int_\Omega \nabla_\Gamma w \cdot \nabla_\Gamma w \dx 
  \leq 0.
\end{equation*}

For the derivation of the weak formulation of~\eqref{eqn:CahnHilliard_bulk} with boundary conditions~\eqref{eqn:CahnHilliard_boundary_GMS} as PDAE, we need to introduce two additional variables on the boundary (for $u|_\Gamma$ and $w|_\Gamma$) and two Lagrange multipliers. 
Hence, we seek for $u, w\colon [0,T] \to \calV$, $p, r\colon [0,T] \to \calVp$, and $\lambda, \mu \colon [0,T] \to \calQ$ such that 
\begin{subequations}
\label{eqn:opEqn:Goldstein}
\begin{alignat}{3}
  \dot u + \sigma\calK_\Omega w - \calB^\ast \mu  
  &= 0 &&\qquad \text{in } \calVw^\ast,\\
  \eps\, \calK_\Omega u + \eps^{-1} W^\prime(u) - \eps\, \calB^\ast \lambda
     &= w &&\qquad \text{in } \calVu^\ast, \\
  \dot p + \calK_\Gamma r + \mu 
  &= 0 &&\qquad \text{in } \calVp^\ast, \\
  \delta \kappa\, \calK_\Gamma p + \delta^{-1} W^\prime_\Gamma(p) + \eps\, \lambda
  &= r &&\qquad \text{in } \calVp^\ast, \\
  \calB u - p 
  &= 0 &&\qquad \text{in } \calQ^\ast, \\
  \calB w - r 
  &= 0 &&\qquad \text{in } \calQ^\ast
\end{alignat}
\end{subequations}
for a.e.~$t\in[0,T]$ and given $u(0)=u^0$, $p(0)=p^0$. As before, this formulation leads to the same conservation and dissipation properties as above. 
%
%
\subsection{Numerical illustration of the different models}\label{sec:CHmodel:illustration} 
We would like to close this section with an illustration of the four types of boundary conditions considered in this paper. First, we summarize the properties such as conservation of mass and dissipation of energy in Table~\ref{tab:conseration}. 
\begin{table}[t]
\caption{Mass-conservation and energy-dissipation properties for the four types of (dynamic) boundary conditions.}
\label{tab:conseration}
\begin{tabular}{l|ccc|ccc}
	boundary conditions & $\int_\Omega u \dx$ & $\int_\Gamma u \dx$ & total mass &$E_{\text{bulk}}$ & $E_{\text{surf}}$ & total energy \\ \hline\hline
	Neumann & yes & -- & -- & yes & -- & -- \\ \hline
	Allen--Cahn type & yes & no & no & no & no & yes	\\[0.2em] 
	Liu--Wu & yes & yes & yes & no & no & yes \\[0.2em] 
	Goldstein et al. & no & no & yes & no & no & yes 
\end{tabular}
\end{table}
Second, we consider a simple test case on the unit square to get a visual idea of the differences. Consider constants~$\eps=\delta=0.02$, $\sigma=\kappa = 1$ and initial data 
\begin{equation*}
  u_0(x,y) 
  = \cos(4 \pi x)\cos(4 \pi y).
\end{equation*}
The simulation results at the final time $T=0.001$ are given in Figure~\ref{fig:comparisonBC}, showing significant differences of the mixture. All computations use the same $P_1$ finite element discretization on a uniform (criss-cross) grid with mesh size $h=0.01$ and time step size~$\tau=T/100$. 
\begin{figure}
\includegraphics[scale=0.5]{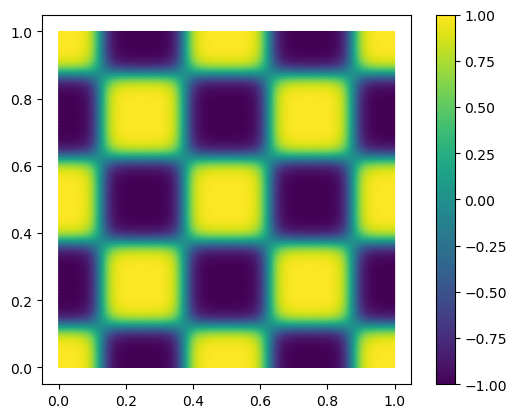}%
\quad
\includegraphics[scale=0.5]{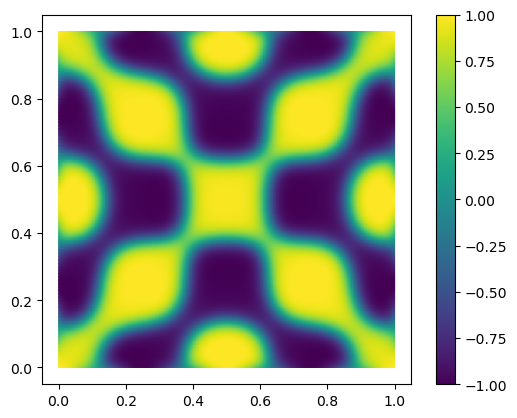}\\
\includegraphics[scale=0.5]{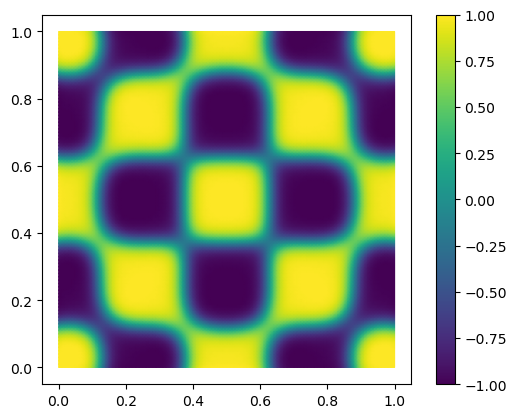}%
\quad
\includegraphics[scale=0.5]{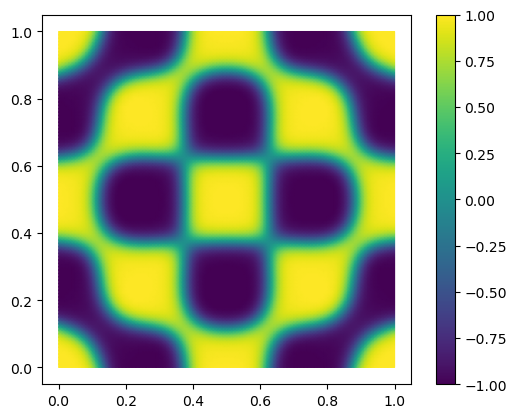}%
\caption{Numerical solutions of the example discussed in Section~\ref{sec:CHmodel:illustration} at time $t=0.001$ with Neumann (top left), Allen--Cahn type (top right), Liu--Wu type (bottom left), and Goldstein--Miranville--Schimpera type boundary conditions (bottom right).}
\label{fig:comparisonBC}
\end{figure}
%
%
\input{firstOrder}
\input{secondOrder}
%
%
\section{Conclusion}\label{sec:conclusion}
In this paper, we have presented PDAE formulations of the Cahn--Hilliard equation with different types of dynamic boundary conditions. Since this approach formally decouples bulk and surface dynamics, different discretizations -- in time and space -- can be chosen in the bulk and on the surface. This increase of flexibility is of particular value if the boundary requires a finer discretization, e.g, due to an oscillatory behaviour of the solution. The proposed time stepping schemes of first and second order preserve the properties of mass-conservation and energy-dissipation and, at the same time, allow a refined spatial discretization of the boundary. 
%
%
\bibliographystyle{alpha} 
\bibliography{../../bib_dynBC}
\end{document}

%% file: def.tex
\newcommand{\V}{\ensuremath{\mathcal{V}}}

\def\N{\mathbb{N}}

\def\R{\mathbb{R}}

\def\eps{\varepsilon}
\definecolor{myBlue1}{RGB}{101,149,239}  
\definecolor{myBlue2}{RGB}{113,104,238} 
\definecolor{myBlue3}{RGB}{30,144,255} 
\definecolor{myGreen1}{RGB}{154,204,50} 
\definecolor{myGreen2}{RGB}{69,169,0} 
\definecolor{myGreen3}{RGB}{154,205,50} 
\definecolor{myGreen4}{RGB}{105,139,34} 
\definecolor{myRed1}{RGB}{210,105,30} 
\definecolor{myRed2}{RGB}{165,42,42} 
\definecolor{myRed3}{RGB}{139,26,26} 
\definecolor{lightgray}{RGB}{175,175,175} 
\definecolor{myLGray}{RGB}{225,225,225} 
\definecolor{mycolor0}{rgb}{0.66,0.66,0.66}
\definecolor{mycolor4}{rgb}{0.00000,0.44700,0.74100}
\definecolor{mycolor1}{rgb}{0.85000,0.32500,0.09800}
\definecolor{mycolor2}{rgb}{0.92900,0.69400,0.12500}
\definecolor{mycolor3}{rgb}{0.67000,0.74700,0.14100}
\definecolor{mycolor5}{rgb}{0.49400,0.18400,0.55600}
\definecolor{mycolor6}{rgb}{0.85000,0.32500,0.09800}%
\newcommand{\calB}{\ensuremath{\mathcal{B}} }

\newcommand{\calK}{\ensuremath{\mathcal{K}} }

\newcommand{\calP}{\ensuremath{\mathcal{P}} }
\newcommand{\calQ}{\ensuremath{\mathcal{Q}} }

\newcommand{\calV}{\ensuremath{\mathcal{V}} }

\def\dx{\,\text{d}x}

\newcommand{\ddt}{\ensuremath{\frac{\text{d}}{\text{d}t}} }

\newcommand{\calVu}{\calV} 
\newcommand{\calVw}{\calV} 
\newcommand{\calVp}{\calP} 

%% file: firstOrder.tex
\section{Dissipation-preserving Discretization of First Order}\label{sec:firstOrder} 
The application of the implicit (or explicit) Euler scheme to one of the Cahn--Hilliard systems from the previous section will, in general, not preserve the proven energy-dissipation. To achieve such a property, we use a decomposition of the nonlinearity into its convex and concave part as suggested in~\cite{Eyr98}. 
Given a function~$f$, we write
\begin{align}
\label{eqn:convexconcaveSplitting}
f(x) 
= f_+(x) - f_-(x)
\end{align}
with~$f_+$ and $f_-$ being convex and, hence, $-f_-$ being concave. Note that this splitting always exists if the Hessian of~$f$ is uniformly bounded
\cite[Th.~1]{YulR03}. The central property for the upcoming proofs reads 
\begin{align}
\label{eqn:propConvexity}
  f_+(y) - f_+(x) \ge f_+'(x)(y-x)
\end{align}
for all $x,y$ in the domain of the convex function~$f_+$. 

Within this section, we focus on the temporal discretization of the introduced PDAEs with constant step size~$\tau$. Hence, we consider a semi-discretization in time only. An additional spatial discretization (e.g.~using finite elements) is straight-forward, since the considered operator formulations correspond to the weak formulation of the system.   
Throughout the proofs, we will make use of the property
\[
  2\, \mathfrak{a}(u, u-v)
  = \Vert u \Vert^2_\mathfrak{a} - \Vert v \Vert^2_\mathfrak{a} + \Vert u- v \Vert^2_\mathfrak{a}
\]
for arbitrary symmetric bilinear forms~$\mathfrak{a}$ and~$\Vert v \Vert^2_\mathfrak{a} \coloneqq \mathfrak{a}(v, v)$. Moreover, we write~$\|\cdot\| \coloneqq \|\cdot\|_{L^2(\Omega)}$ and $\|\cdot\|_\Gamma \coloneqq \|\cdot\|_{L^2(\Gamma)}$ for the respective $L^2$-norms in $\Omega$ and on~$\Gamma$. 

Before we deal with dynamic boundary conditions, we would like to comment on the situation for homogeneous Neumann boundary conditions. Considering the PDAE~\eqref{eqn:opEqn:Neumann} and a convex--concave splitting of the potential~$W$, we obtain the time stepping scheme 
\begin{subequations}
\label{eqn:firstOrder:Neumann}
\begin{alignat}{3}
  u^{n+1} + \tau\, \sigma \calK_\Omega w^{n+1} 
  &= u^n &&\qquad \text{in } \calVw^\ast, \label{eqn:firstOrder:Neumann:a}\\
  \eps\, \calK_\Omega u^{n+1} + \eps^{-1} \big[ W_+^\prime(u^{n+1}) - W_-^\prime(u^n) \big] 
  &= w^{n+1} &&\qquad \text{in } \calVu^\ast \label{eqn:firstOrder:Neumann:b}
\end{alignat}
\end{subequations}
with a given starting value~$u^0\in \V$. Note that this equals the implicit Euler scheme with the modification that the derivative of the concave part of $W$ is handled explicitly. For this scheme and a sufficiently smooth potential~$W$, one can show first-order accuracy and energy-dissipation, i.e., 
\[
  E_{\text{bulk}}(u^{n})
  \ge E_{\text{bulk}}(u^{n+1})
\]
for all $n\ge 0$. 
To see the latter, one considers test functions~$w^{n+1}\in\V$ in equation~\eqref{eqn:firstOrder:Neumann:a} and $u^{n+1}-u^n\in\V$ in~\eqref{eqn:firstOrder:Neumann:b}. Property~\eqref{eqn:propConvexity}, which reads here 
\begin{align*}
  W_+^\prime(u^{n+1})\, (u^{n+1}-u^n)
  &\ge W_+(u^{n+1}) - W_+(u^{n}), \\
  W_-^\prime(u^{n})\, (u^{n+1}-u^n)
  &\le W_-(u^{n+1}) - W_-(u^{n}), 
\end{align*}
then leads to the claimed dissipativity of~$E_{\text{bulk}}$. 

In the following three subsections, we show that the convex--concave splitting is also applicable for non-standard boundary conditions. Moreover, we discuss the possibility of applying smaller time steps on the boundary. 
%
%
\subsection{Allen--Cahn type boundary conditions}\label{sec:firstOrder:AllenCahn} 
We turn to the dynamic boundary conditions of Allen--Cahn type introduced in Section~\ref{sec:CHmodel:AllenCahn}. Following the idea of the convex--concave splitting of the potential~$W$, the discretization of~\eqref{eqn:opEqn:AllenCahn} yields the system 
\begin{subequations}
\label{eqn:firstOrder:AllenCahn}
\begin{alignat}{3}
  u^{n+1} + \tau\, \sigma\calK_\Omega w^{n+1} 
  &= u^n &&\quad \text{in } \calVw^\ast, \label{eqn:firstOrder:AllenCahn:a}\\
  \eps\, \calK_\Omega u^{n+1} + \eps^{-1} \big[ W_+^\prime(u^{n+1}) - W_-^\prime(u^n) \big] - \eps\, \calB^\ast \lambda^{n+1} 
  &= w^{n+1} &&\quad \text{in } \calVu^\ast, \label{eqn:firstOrder:AllenCahn:b}\\
  p^{n+1} + \tau \delta \kappa\, \calK_\Gamma p^{n+1} + \tau\delta^{-1} \big[ W_{\Gamma,+}^\prime(p^{n+1}) - W_{\Gamma,-}^\prime(p^n) \big] + \tau\eps\,  \lambda^{n+1} 
  &= p^n &&\quad \text{in } \calVp^\ast, \label{eqn:firstOrder:AllenCahn:c}\\
  \calB u^{n+1} - p^{n+1}  
  &= 0 &&\quad \text{in } \calQ^\ast. \label{eqn:firstOrder:AllenCahn:d}
\end{alignat}
\end{subequations}
Corresponding (consistent) initial data is given by $u^0\in\calVu$ and $p^0\in\calVp$. Note that neither the chemical potential~$w$ nor the Lagrange multiplier~$\lambda$ need an initial value. 
\begin{proposition}
\label{prop:firstOrder:AllenCahn}
Assume consistent initial data, i.e., $p^0 = u^0|_\Gamma$. Then, the scheme~\eqref{eqn:firstOrder:AllenCahn} is first-order accurate and energy-dissipative, i.e., 
\[
  E(u^{n}) \ge E(u^{n+1})
\]
for all $n\ge 0$. 
\end{proposition}
\begin{proof}
The first-order accuracy of the method (for $W\in C^{2}(\R)$) follows from the fact that this equals the implicit Euler scheme up to the transition from~$W_-^\prime(u^{n+1})$ to $W_-^\prime(u^n)$, cf.~\cite[Th.~4.1]{ChePP10}. This, however, only depicts a perturbation of order~$\tau$.  	
We turn to the dissipativity property. By the consistency assumption and equation~\eqref{eqn:firstOrder:AllenCahn:d}, we know that $p^n = u^n|_\Gamma$ for all $n\ge0$. Now consider~\eqref{eqn:firstOrder:AllenCahn:a}-\eqref{eqn:firstOrder:AllenCahn:c} with test functions~$w^{n+1}$, $u^{n+1}-u^n$, and~$\frac1\tau(p^{n+1}-p^n) = \frac 1\tau (u^{n+1}-u^n)|_\Gamma$, respectively. The sum of these three equations gives 
\begin{align*}
0
&= \frac\eps2\, \big( \|\nabla u^{n+1}\|^2 - \|\nabla u^{n}\|^2 + \|\nabla (u^{n+1}-u^n)\|^2 \big) + \tau\sigma\, \| \nabla w^{n+1}\|^2 + \tau^{-1} \|p^{n+1} - p^n\|^2_\Gamma \\ 
&\quad + \frac{\delta\kappa}{2}\, \big( \|\nabla_\Gamma p^{n+1}\|_\Gamma^2 - \|\nabla_\Gamma p^{n}\|_\Gamma^2 + \|\nabla_\Gamma (p^{n+1}-p^n)\|_\Gamma^2 \big) \\
&\quad+ \eps^{-1} \big( W_+^\prime(u^{n+1}) - W_-^\prime(u^n) \big)\, (u^{n+1}-u^n) + \delta^{-1} \big( W_{\Gamma,+}^\prime(p^{n+1}) - W_{\Gamma,-}^\prime(p^n) \big)(p^{n+1}-p^n).
\end{align*}
Several applications of the convexity property~\eqref{eqn:propConvexity} then yields  
\begin{align*}
0
&\ge \frac\eps2\, \big( \|\nabla u^{n+1}\|^2 - \|\nabla u^{n}\|^2 \big) + \frac{\delta\kappa}{2}\, \big( \|\nabla_\Gamma p^{n+1}\|_\Gamma^2 - \|\nabla_\Gamma p^{n}\|_\Gamma^2 \big) \\*
&\hspace{3.3cm}+ \eps^{-1} \big( W(u^{n+1}) - W(u^{n}) \big) + \delta^{-1} \big( W_{\Gamma}(p^{n+1}) - W_{\Gamma}(p^n) \big) \\
&= E(u^{n+1}) - E(u^{n}),
\end{align*}
which completes the proof. 
\end{proof}
A numerical experiment validating the above result, is given in Section~\ref{sec:firstOrder:numerics}. Therein, we consider the situation where $\delta$ is smaller than $\eps$, calling for a fine discretization (in time and space) on the boundary. As we will discuss later on, the formulation as PDAE with the auxiliary variable $p$ allows to consider finer spatial discretizations on the boundary without any additional effort. Moreover, we show in the sequel how to implement a refined discretization in time as well. 

In order to allow a smaller time step size for the computation on the boundary, we introduce the parameter $\ell\in\N$ and consider time steps of size $\tau/\ell$; see the illustration in Figure~\ref{fig:boundaryRefinement}. 
%
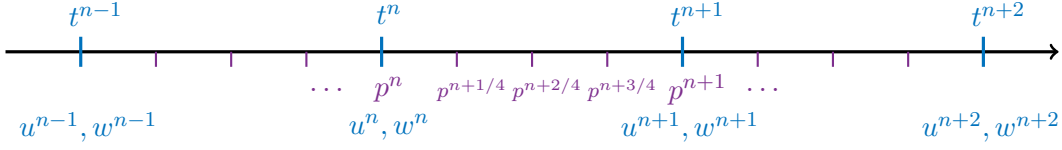
\begin{figure}
\centering
\begin{tikzpicture}[scale=1]¸
	\draw[very thick,black,->] (-1,0) -- (13,0);
	\foreach \x in {0,4,8,12} {		
		\draw[mycolor4,very thick] (\x, -0.2) -- (\x, 0.2);
	}
	\node[mycolor4] at (0.2,.5) {$t^{n-1}$};
	\node[mycolor4] at (4.1,.5) {$t^{n}$};
	\node[mycolor4] at (8.2,.5) {$t^{n+1}$};
	\node[mycolor4] at (12.2,.5) {$t^{n+2}$};	
    \node[mycolor5] at (3.3,-.5) {$\dots$};
	\node[mycolor5] at (4.1,-.5) {$p^{n}$};
	\node[mycolor5] at (8.2,-.5) {$p^{n+1}$};
    \node[mycolor5] at (9.1,-.5) {$\dots$};
	\node[mycolor4] at (0.1,-1) {$u^{n-1}, w^{n-1}$};
	\node[mycolor4] at (4.1,-1) {$u^{n}, w^n$};
	\node[mycolor4] at (8.1,-1) {$u^{n+1}, w^{n+1}$};
	\node[mycolor4] at (12.1,-1) {$u^{n+2}, w^{n+2}$};	
	\foreach \x in {1,2,3,5,6,7,9,10,11} {		
		\draw[mycolor5,thick] (\x, -0.2) -- (\x, 0);
	}
	\node[mycolor5] at (5.2,-.5) {\scriptsize $p^{n+1/4}$};
	\node[mycolor5] at (6.2,-.5) {\scriptsize $p^{n+2/4}$};
	\node[mycolor5] at (7.2,-.5) {\scriptsize $p^{n+3/4}$};
\end{tikzpicture}
\caption{Illustration of the time mesh for $\ell=4$. The variables $u$ and $w$ are only computed at the time points $t^n=n\tau$, whereas $p$ on the boundary is computed on a refined time grid with step size $\tau/\ell$. }
\label{fig:boundaryRefinement}
\end{figure}  
Then, we replace~\eqref{eqn:firstOrder:AllenCahn:c} by the~$\ell$ equations
\begin{align}
  p^{n+\frac{j}{\ell}} + \tfrac \tau \ell\, \delta \kappa\, \calK_\Gamma p^{n+\frac{j}{\ell}} + \tfrac \tau \ell\, \delta^{-1} \big[ W_{\Gamma,+}^\prime(p^{n+\frac{j}{\ell}}) - W_{\Gamma,-}^\prime(p^{n+\frac{j-1}{\ell}}) \big] + \tfrac \tau \ell\, \eps\, \lambda^{n+1} 
  = p^{n+\frac{j-1}{\ell}} 
  \label{eqn:firstOrder:AllenCahn:cRefined}
\end{align}
stated in $\calVp^\ast$ for $j=1,\dots,\ell$. 

Consider a spatial discretization with $N_u$, $N_p$, and $N_\lambda$ degrees of freedom for the variables $u$, $p$, and $\lambda$, respectively. Then, this extension leads to a nonlinear system of size $2N_u+\ell N_p+N_\lambda$. Since we expect $N_p \ll N_u$ (we have, e.g., $N_u=N_p^2$ for a uniform grid in two space dimensions), this is much smaller than considering~\eqref{eqn:firstOrder:AllenCahn:c} entirely with the refined time step size $\tau/\ell$. 
\begin{remark}
The numerical scheme with \eqref{eqn:firstOrder:AllenCahn:c} replaced by \eqref{eqn:firstOrder:AllenCahn:cRefined} is still energy-dissipative. To see this, one considers test functions $\frac\ell\tau(p^{n+\frac{j}{\ell}}-p^{n+\frac{j-1}{\ell}})$ and proceeds as in the proof of Proposition~\ref{prop:firstOrder:AllenCahn}. 
\end{remark}
In the following two subsections, we turn to dynamic boundary conditions of Cahn--Hilliard type. 
%
%
\subsection{Boundary conditions of Liu and Wu}\label{sec:firstOrder:LiuWu} 
We start with the boundary conditions discussed in Section~\ref{sec:CHmodel:LiuWu}. The first-order discretization of~\eqref{eqn:opEqn:LiuWu} using the convex--concave splitting yields the time stepping scheme 
\begin{subequations}
\label{eqn:firstOrder:LiuWu}
\begin{alignat}{3}
	u^{n+1} + \tau\,\sigma \calK_\Omega w^{n+1} 
	&= u^n &&\qquad \text{in } \calVw^\ast, \label{eqn:firstOrder:LiuWu:a}\\
	\eps\, \calK_\Omega u^{n+1} + \eps^{-1} \big[ W_+^\prime(u^{n+1}) - W_-^\prime(u^n) \big] - \eps\, \calB^\ast \lambda^{n+1} 
	&= w^{n+1} &&\qquad \text{in } \calVu^\ast, \label{eqn:firstOrder:LiuWu:b}\\
	p^{n+1} + \tau\, \calK_\Gamma w_\Gamma^{n+1} 
	&= p^n &&\qquad \text{in } \calVp^\ast, \label{eqn:firstOrder:LiuWu:c}\\
	\delta \kappa\, \calK_\Gamma p^{n+1} + \delta^{-1} \big[ W_{\Gamma,+}^\prime(p^{n+1}) - W_{\Gamma,-}^\prime(p^n) \big] + \eps\,  \lambda^{n+1} 
	&= w_\Gamma^{n+1} &&\qquad \text{in } \calVp^\ast, \label{eqn:firstOrder:LiuWu:d}\\
	\calB u^{n+1} - p^{n+1}
	&= 0 &&\qquad \text{in } \calQ^\ast.
\end{alignat}
\end{subequations}
As in the previous section, we expect initial data $u^0\in\calVu$ and $p^0\in\calVp$. The convex--concave splitting again guarantees the preservation of energy-dissipation. 
\begin{proposition}
\label{prop:firstOrder:LiuWu}	
Under the assumption of consistent initial data, i.e., $p^0 = u^0|_\Gamma$, the scheme~\eqref{eqn:firstOrder:LiuWu} is first-order accurate and energy-dissipative, i.e., $E(u^{n})	\ge E(u^{n+1})$ for all $n\ge 0$. 
\end{proposition}
\begin{proof}
We only show the dissipation of the total energy. For this, we consider~\eqref{eqn:firstOrder:LiuWu:a}-\eqref{eqn:firstOrder:LiuWu:d} with test functions~$w^{n+1}$, $u^{n+1}-u^n$, $w_\Gamma^{n+1}$, and~$p^{n+1}-p^n = (u^{n+1}-u^n)|_\Gamma$, respectively. The sum of these equations gives 
\begin{align*}
0
&= \frac\eps2\, \big( \|\nabla u^{n+1}\|^2 - \|\nabla u^{n}\|^2 + \|\nabla (u^{n+1}-u^n)\|^2 \big) + \tau\sigma\, \| \nabla w^{n+1}\|^2 \\
&\quad+ \frac{\delta\kappa}{2}\, \big( \|\nabla_\Gamma p^{n+1}\|_\Gamma^2 - \|\nabla_\Gamma p^{n}\|_\Gamma^2 + \|\nabla_\Gamma (p^{n+1}-p^n)\|_\Gamma^2 \big) + \tau\, \| \nabla w_\Gamma^{n+1}\|_\Gamma^2 \\
&\quad+ \eps^{-1} \big( W_+^\prime(u^{n+1}) - W_-^\prime(u^n) \big)\, (u^{n+1}-u^n) + \delta^{-1} \big( W_{\Gamma,+}^\prime(p^{n+1}) - W_{\Gamma,-}^\prime(p^n) \big)(p^{n+1}-p^n).
\end{align*}
Using again the convexity property~\eqref{eqn:propConvexity}, we obtain $0\ge E(u^{n+1})-E(u^n)$.
%
\end{proof}
\begin{remark}
As for the boundary conditions of Allen--Cahn type, we may consider a finer discretization on the boundary. For this, one replaces equations~\eqref{eqn:firstOrder:LiuWu:c} and~\eqref{eqn:firstOrder:LiuWu:d} by an appropriate discretization with step size $\tau/\ell$; see the construction in Section~\ref{sec:firstOrder:AllenCahn}.	This then again maintains the dissipation property of the energy. 
\end{remark}
%
%
\subsection{Boundary conditions of Goldstein, Miranville, and Schimpera}\label{sec:firstOrder:Goldstein} 
Finally, we consider the second model of dynamic boundary conditions of Cahn--Hilliard type. Here, the discretization of~\eqref{eqn:opEqn:Goldstein} using the convex--concave splitting yields the time stepping scheme 
\begin{subequations}
\label{eqn:firstOrder:GMS}
\begin{alignat}{3}
	u^{n+1} + \tau\,\sigma \calK_\Omega w^{n+1} - \tau\, \calB^\ast \mu^{n+1} 
	&= u^n &&\qquad \text{in } \calVw^\ast, \label{eqn:firstOrder:GMS:a}\\
	\eps\, \calK_\Omega u^{n+1} + \eps^{-1} \big[ W_+^\prime(u^{n+1}) - W_-^\prime(u^n) \big] - \eps\, \calB^\ast \lambda^{n+1} 
	&= w^{n+1} &&\qquad \text{in } \calVu^\ast, \label{eqn:firstOrder:GMS:b}\\
	p^{n+1} + \tau\, \calK_\Gamma r^{n+1} + \tau \mu^{n+1} 
	&= p^n &&\qquad \text{in } \calVp^\ast, \label{eqn:firstOrder:GMS:c}\\
	\delta \kappa\, \calK_\Gamma p^{n+1} + \delta^{-1} \big[ W_{\Gamma,+}^\prime(p^{n+1}) - W_{\Gamma,-}^\prime(p^n) \big] + \eps\,  \lambda^{n+1} 
	&= r^{n+1} &&\qquad \text{in } \calVp^\ast, \label{eqn:firstOrder:GMS:d}\\
	\calB u^{n+1} - p^{n+1}
	&= 0 &&\qquad \text{in } \calQ^\ast, \\
	\calB w^{n+1} - r^{n+1}
	&= 0 &&\qquad \text{in } \calQ^\ast
\end{alignat}
\end{subequations}
with initial data $u^0\in\calVu$ and $p^0\in\calVp$. Similar to the previous two models, one may introduce a temporal refinement on the boundary by an adjustment of equations~\eqref{eqn:firstOrder:GMS:c} and~\eqref{eqn:firstOrder:GMS:d}. In any case, we get the following result on the dissipation of energy. 
\begin{proposition}
\label{prop:firstOrder:GMS}
Under the assumption of consistent initial data, i.e., $p^0 = u^0|_\Gamma$, the scheme~\eqref{eqn:firstOrder:GMS} is first-order accurate and energy-dissipative. 
\end{proposition}
\begin{proof}
For the dissipation property, we consider~\eqref{eqn:firstOrder:GMS:a}-\eqref{eqn:firstOrder:GMS:d} with test functions~$w^{n+1}$, $u^{n+1}-u^n$, $r^{n+1}=w^{n+1}|_\Gamma$, and~$p^{n+1}-p^n = (u^{n+1}-u^n)|_\Gamma$, respectively. 
A calculation as for the Liu--Wu model in Proposition~\ref{prop:firstOrder:LiuWu} then yields the assertion~$E(u^{n}) \ge E(u^{n+1})$.
\end{proof}
%
%
\subsection{Numerical experiments}\label{sec:firstOrder:numerics} 
In this final part on first-order discretization schemes, we illustrate the claimed dissipation properties and the impact of possible refinements on the boundary for two model problems. 
%
\subsubsection{Example with Allen--Cahn type boundary conditions}
\label{sec:firstOrder:numerics:AC} 
We consider dynamic boundary conditions of Allen--Cahn type with the time stepping scheme introduced in Section~\ref{sec:firstOrder:AllenCahn}. For this, we choose the unit square as spatial domain and $u^0 = \cos(4\pi x)\cos(4\pi y)$, $p^0=u^0|_\Gamma$ as consistent initial values. The interaction lengths are given by $\eps = 0.02$ and $\delta = 0.2$, the dissipation coefficients by~$\sigma = 0.01$ and $\kappa = 5$. Figure~\ref{fig:error_4_ell} shows the errors in~$u$ and~$p$ for different numbers of intermediate time steps on the boundary (characterized by the parameter~$\ell$). Here, the time horizon is $T=0.2$ and the time step size~$\tau = 0.1\cdot 2^{-7}$. We observe that the additional time steps reduce the error for~$p$, whereas the error for~$u$ remains unchanged. For the sake of completeness, we emphasize that the total energy decreases for any choice of~$\ell$ as predicted in Proposition~\ref{prop:firstOrder:AllenCahn}.
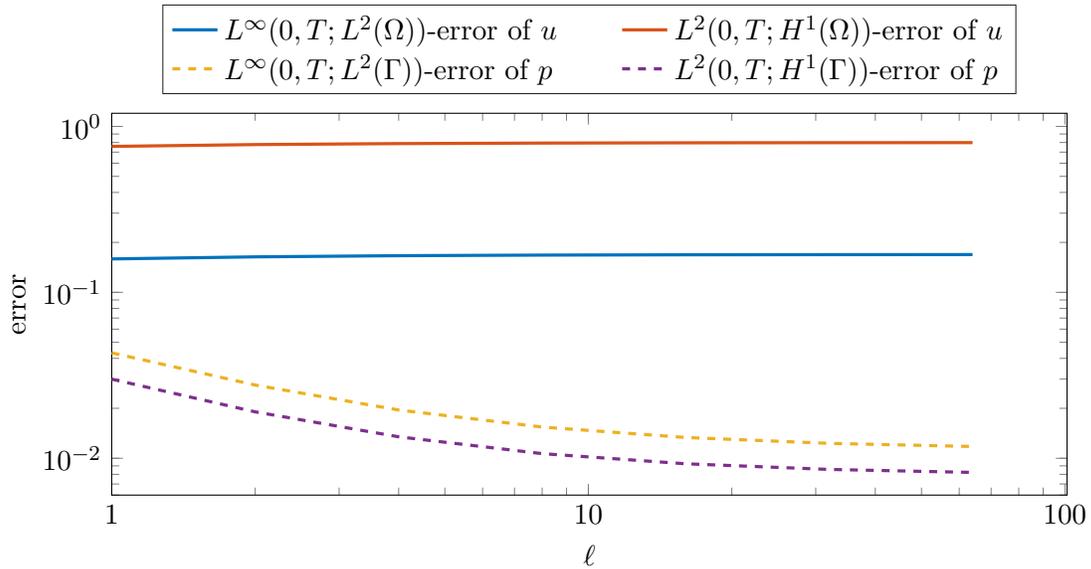
\begin{figure}
\input{pics/error_p}
\caption{Errors for different numbers of intermediate time steps on the boundary (with step size $\tau/\ell$) for the example of Section~\ref{sec:firstOrder:numerics:AC}.  }
\label{fig:error_4_ell}
\end{figure}

In this example, we have chosen the dissipation coefficient~$\kappa$ in such a way  that the solution~$p$ changes rapidly on the boundary. While for the here considered Allen--Cahn type boundary conditions this leads to a vanishing~$p$, for boundary conditions of Liu--Wu type, the phases on the boundary separate quickly. In that setting, a finer temporal mesh for the variable~$p$ has an even bigger impact as we show in the second example. 
%
\subsubsection{Example with Liu--Wu type boundary conditions}
\label{sec:firstOrder:numerics:CH} 
We now consider Liu--Wu type boundary conditions and the numerical scheme from Section~\ref{sec:firstOrder:LiuWu}. We choose the same parameters as in the previous subsection but with~$\kappa=10$. As mentioned before, this leads to a rapidly changing variable $p$. Due to the constraint~$u|_{\Gamma} = p$, this also drives the dynamic behavior of the solution in the bulk. As a result, a refined time discretization of~$p$ also has a positive effect on the approximation of~$u$, which can be seen in Figure~\ref{fig:CH_CH_ell_final_time}.
\begin{figure}
\includegraphics[width=4cm,height=4cm]{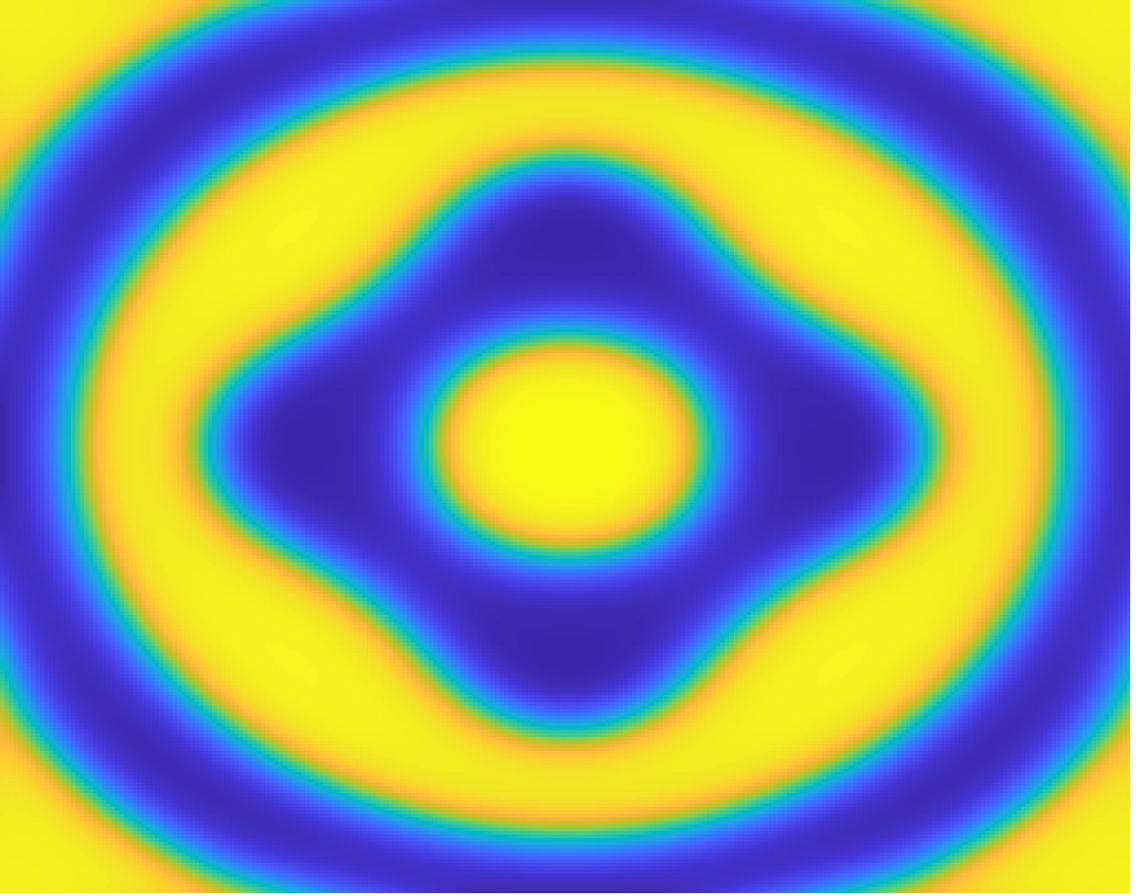}\qquad
\includegraphics[width=4cm,height=4cm]{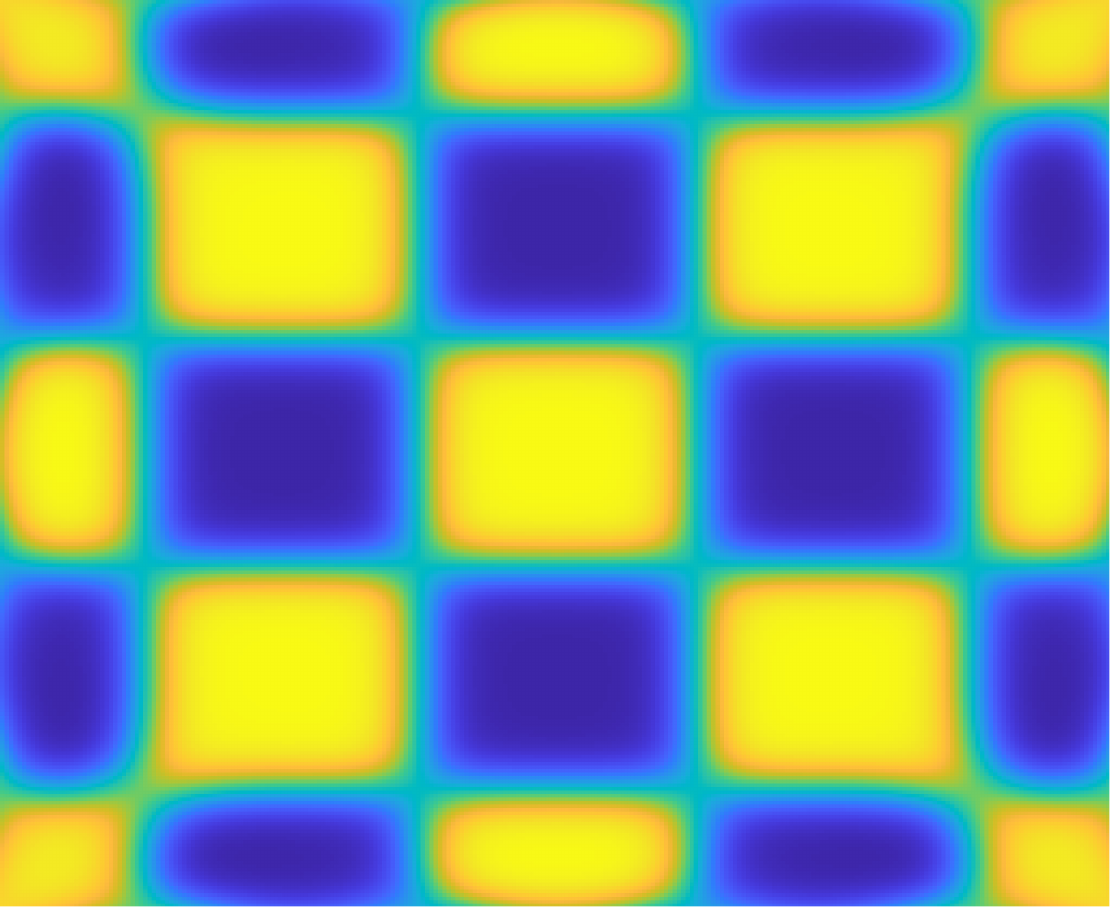}\qquad
\includegraphics[width=4cm,height=4cm]{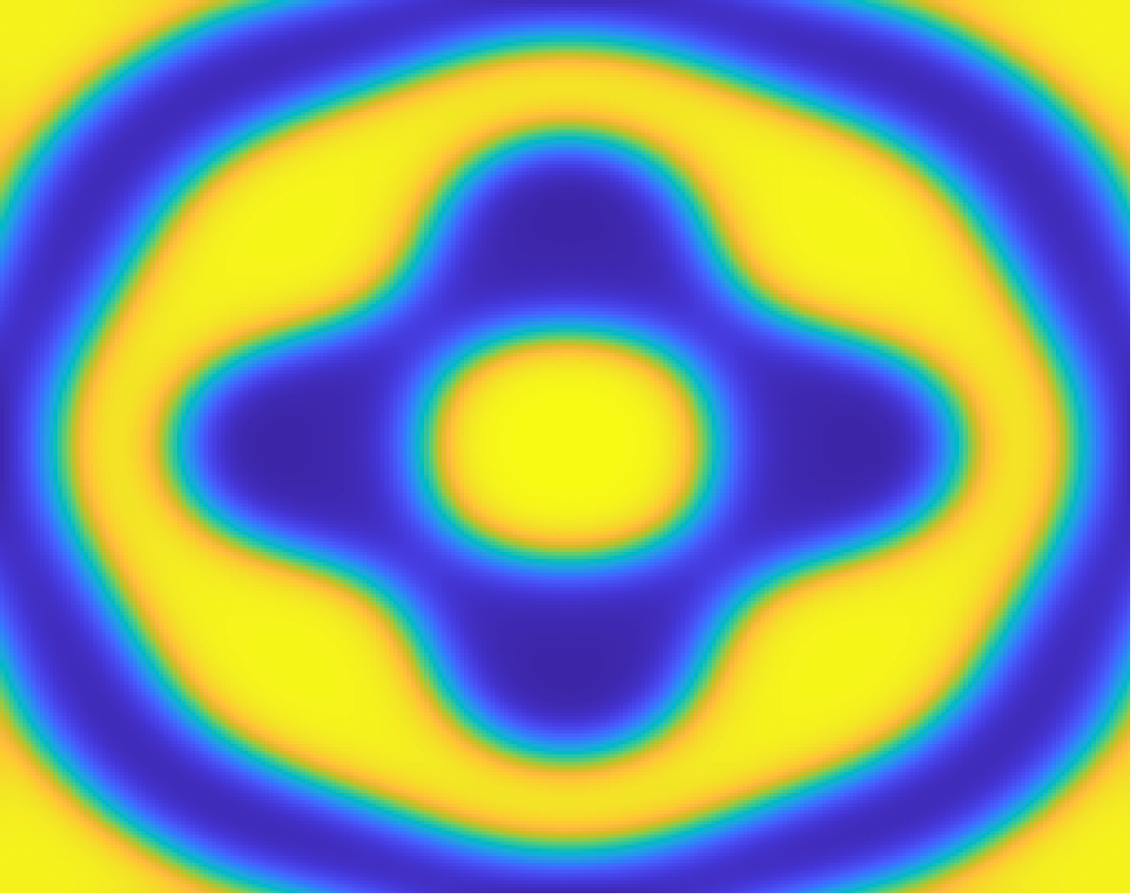}
\caption{Numerical solutions of the Cahn--Hilliard equation with Liu--Wu type boundary conditions at $t=0.1$ for different temporal step sizes: $\tau=1.5625\cdot 10^{-5}$, $\ell=1$ (left), $\tau=2\cdot 10^{-3}$, $\ell=1$ (middle), and $\tau=2\cdot 10^{-3}$, $\ell=16$ (right).}
\label{fig:CH_CH_ell_final_time}
\end{figure}
Obviously, the temporal discretization in the middle is too coarse such that the behavior of~$p$ is not well reproduced. With some additional time steps on the boundary, however, the reference solution (left) and its numerical approximation (right) are much closer. This can also be observed in the errors shown in Figure~\ref{fig:CH_CH_error_ell}.
\begin{figure}
\input{pics/CH_CH_error_ell}
\caption{Errors for $\tau=2\cdot 10^{-3}$ and different numbers of intermediate time steps (with step size $\tau/\ell$) on the boundary for the example of Section~\ref{sec:firstOrder:numerics:CH}.}
\label{fig:CH_CH_error_ell}
\end{figure}
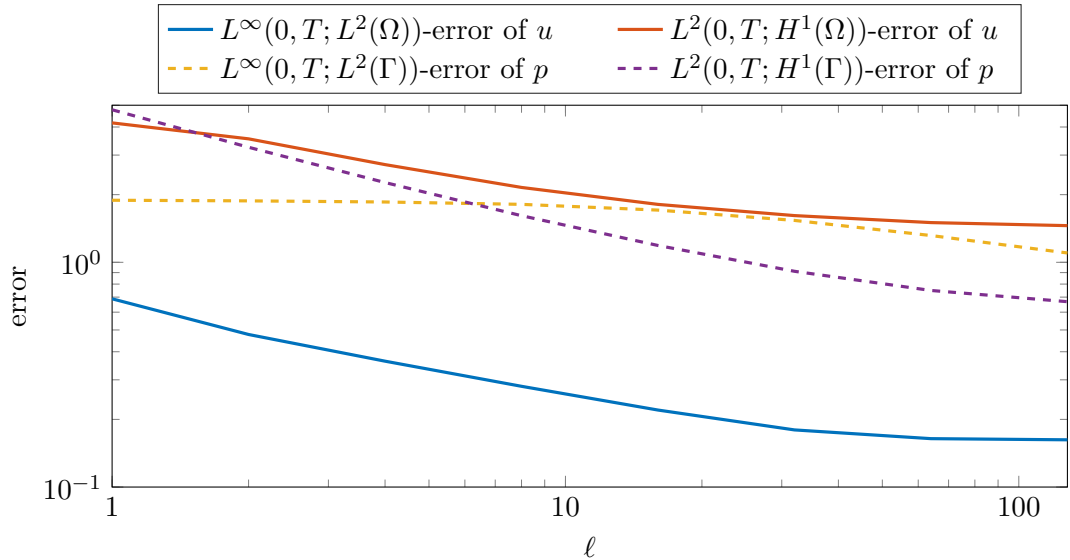

%% file: pics/error_p.tex
%
%
\definecolor{mycolor1}{rgb}{0.00000,0.44700,0.74100}%
\definecolor{mycolor2}{rgb}{0.85000,0.32500,0.09800}%
\definecolor{mycolor3}{rgb}{0.92900,0.69400,0.12500}%
\definecolor{mycolor4}{rgb}{0.49400,0.18400,0.55600}%
\begin{tikzpicture}

\begin{axis}[%
width=5.0in,
height=2.0in,
at={(2.502in,1.101in)},
scale only axis,
xmode=log,
xmin=1,
xmax=101,
xtick={1,10,100},
xticklabels={$1$,$10$,$100$},
xlabel=$\ell$,
ymode=log,
ymin=0.006,
ymax=1.2,
yminorticks=true,
ylabel=error,
axis background/.style={fill=white},
legend columns = 2,
legend style={legend cell align=left, align=left, at={(0.5,1.05)}, anchor=south, draw=white!15!black}
]
\addplot [color=mycolor1, very thick]
  table[row sep=crcr]{%
1	0.158994313263448\\
2	0.163607757841209\\
4	0.166286216841623\\
8	0.16766756339135\\
16	0.168368943047456\\
32	0.168722330757982\\
64	0.168899701693298\\
};
\addlegendentry{$L^\infty(0,T;L^2(\Omega))$-error of $u$\qquad}

\addplot [color=mycolor2, very thick]
  table[row sep=crcr]{%
1	0.757749167883295\\
2	0.777211852561135\\
4	0.787927581596075\\
8	0.793566680952436\\
16	0.796461701114488\\
32	0.797928781467317\\
64	0.798667305989666\\
};
\addlegendentry{$L^2(0,T;H^1(\Omega))$-error of $u$}

\addplot [color=mycolor3, very thick, dashed]
  table[row sep=crcr]{%
1	0.0431841593830126\\
2	0.0276228254972625\\
4	0.0195424029640684\\
8	0.0154232595808663\\
16	0.0133434487995587\\
32	0.0122984242797508\\
64	0.0117746255974231\\
};
\addlegendentry{$L^\infty(0,T;L^2(\Gamma))$-error of $p$\qquad}

\addplot [color=mycolor4, very thick, dashed]
  table[row sep=crcr]{%
1	0.0299821550677144\\
2	0.0190478660347389\\
4	0.0134699754314019\\
8	0.0106654916963906\\
16	0.0092649527474169\\
32	0.00856671448839346\\
64	0.00821843516411079\\
};
\addlegendentry{$L^2(0,T;H^1(\Gamma))$-error of $p$}

\end{axis}
\end{tikzpicture}%

%% file: pics/CH_CH_error_ell.tex
%
%
\definecolor{mycolor1}{rgb}{0.00000,0.44700,0.74100}%
\definecolor{mycolor2}{rgb}{0.85000,0.32500,0.09800}%
\definecolor{mycolor3}{rgb}{0.92900,0.69400,0.12500}%
\definecolor{mycolor4}{rgb}{0.49400,0.18400,0.55600}%
\begin{tikzpicture}

\begin{axis}[%
width=5.0in,
height=2.0in,
at={(0.758in,0.481in)},
scale only axis,
xmode=log,
xmin=1,
xmax=128,
xtick={1,10,100},
xticklabels={$1$,$10$,$100$},
xminorticks=true,
xlabel=$\ell$,
ymode=log,
ymin=0.1,
ymax=5,
yminorticks=true,
ylabel=error,
legend columns = 2,
legend style={legend cell align=left, align=left, at={(0.5,1.03)}, anchor=south, draw=white!15!black}
]
\addplot [color=mycolor1, very thick]
  table[row sep=crcr]{%
1	0.688012332287056\\
2	0.477397154246698\\
4	0.363186693942521\\
8	0.280715683203514\\
16	0.220106710941198\\
32	0.179585754366018\\
64	0.164202513397434\\
128 0.162270641480676\\
};
\addlegendentry{$L^\infty(0,T;L^2(\Omega))$-error of $u$\qquad }

\addplot [color=mycolor2, very thick]
  table[row sep=crcr]{%
1	4.17245963287402\\
2	3.54328460612032\\
4	2.72132268230312\\
8	2.15355853364285\\
16	1.81024525003219\\
32	1.61396578277007\\
64  1.502411728898540\\
128 1.455983582329573\\
};
\addlegendentry{$L^2(0,T;H^1(\Omega))$-error of $u$ }

\addplot [color=mycolor3, very thick, dashed]
  table[row sep=crcr]{%
1	1.88890071527927\\
2	1.87705515782608\\
4	1.85531446769663\\
8	1.8111665050211\\
16	1.70732845042349\\
32	1.53634807621663\\
64 1.315799595446598\\
128 1.100472025853220\\
};
\addlegendentry{$L^\infty(0,T;L^2(\Gamma))$-error of $p$\qquad }

\addplot [color=mycolor4, very thick, dashed]
  table[row sep=crcr]{%
1	4.77255210023299\\
2	3.25066315562108\\
4	2.26440953062336\\
8	1.61229984001709\\
16	1.1881889081304\\
32	0.912126885748158\\
64  0.749526086163086\\
128 0.668104020282098\\
};
\addlegendentry{$L^2(0,T;H^1(\Gamma))$-error of $p$}
\end{axis}


\end{tikzpicture}%

%% file: secondOrder.tex
\section{Dissipation-preserving Discretization of Second Order}\label{sec:secondOrder} 
This section is devoted to the construction of second-order schemes, which are dissipation-preserving. Here, a convex--concave splitting of the potentials $W$ and $W_\Gamma$ is not sufficient, since this limits the convergence order to one. Instead, we consider a discretization of Crank--Nicolson type. 
In this section, we restrict ourselves to the case of polynomial double-well potentials, i.e., we consider the nonlinearities 
\[
  W(u) = \frac14\, (u^2-1)^2,\qquad
  W_\Gamma(p) = \frac14\, (p^2-1)^2.
\]
%

As in the previous section, we focus on the temporal discretization, i.e., we discuss time stepping schemes for the operator formulations presented in Section~\ref{sec:CHmodel} with constant step size~$\tau$. 
In the numerical experiment, however, we will also illustrate the possibility of using different spatial discretizations in the bulk and on the surface. 
Throughout this section, we need averages of the previous and the current iteration. To shorten notation, we hence introduce $u^{n+1/2} \coloneqq \frac12 (u^{n} + u^{n+1})$ and analogously for the other variables. 

As in the previous section, we introduce the dissipation-preserving time stepping scheme by means of the pure Neumann case, i.e., by system~\eqref{eqn:opEqn:Neumann}. With the specific choice of the nonlinearity, the proposed time stepping scheme reads 
\begin{subequations}
\label{eqn:secondOrder:fullCN}
\begin{alignat}{3}
	u^{n+1} + \tau\,\sigma \calK_\Omega w^{n+1/2} 
	&= u^n &&\qquad \text{in } \calVw^\ast, \label{eqn:secondOrder:fullCN:a}\\
	\eps\, \calK_\Omega u^{n+1/2} + \eps^{-1} \big( \tfrac{|u^{n+1}|^2+|u^n|^2}{2} - 1 \big)\, u^{n+1/2} 
	&= w^{n+1/2} &&\qquad \text{in } \calVu^\ast. \label{eqn:secondOrder:fullCN:b}
\end{alignat}
\end{subequations}
Note that, in contrast to the classical second-order Crank--Nicolson scheme, we use the expression~$\frac12 ( |u^{n+1}|^2+|u^n|^2 )$ instead of~$|u^{n+1/2}|^2$ within the nonlinearity, which is itself a second-order perturbation; see also~\cite{DuN91,Ell89}. At this point, we would like to emphasize that the classical Crank--Nicolson scheme is, in general, not energy-dissipative. 
\begin{remark}[Initial value of $w$]\label{rem:w0}
In contrast to the discretizations discussed in Section~\ref{sec:firstOrder}, the time integration scheme~\eqref{eqn:secondOrder:fullCN} calls for an initial value~$w^0$ of the variable~$w$. This can be calculated by fixing $u^0$ and solving~\eqref{eqn:opEqn:Neumann} for $w(0)$ and $\dot{u}(0)$. 
For the systems with dynamical boundary conditions, one calculates the initial values $w^0$, $w^0_{\Gamma}$, and $r^0$ analogously by considering the corresponding continuous system at time~$t=0$. 
\end{remark}
%
Considering~\eqref{eqn:secondOrder:fullCN:a} with the constant function as test function, one observes that this scheme maintains the conservation of mass property from the continuous setting. 
Also the dissipation property of the energy is preserved, i.e., it holds that~$E_{\text{bulk}}(u^n) \ge E_{\text{bulk}}(u^{n+1})$ for all $n\ge 0$. To see this, we consider the sum of~\eqref{eqn:secondOrder:fullCN:a}, tested with $w^{n+1/2}$, and~\eqref{eqn:secondOrder:fullCN:b}, tested with~$u^{n+1}-u^n$. This yields 
\[
  \eps\, \langle \calK_\Omega u^{n+1/2}, u^{n+1}-u^n \rangle 
  + \tfrac1{2\eps}\, \langle (|u^{n+1}|^2+|u^n|^2 -2)\, u^{n+1/2}, u^{n+1}-u^n\rangle 
  + \tau\sigma\, \| \nabla w^{n+1/2} \|^2
  = 0. 
\]
The first term on the left-hand side equals $\frac\eps2\, \|\nabla u^{n+1}\|^2 - \frac\eps2\, \|\nabla u^n \|^2$, whereas for the second term we use that 
\[
  2\, \big\langle (|u^{n+1}|^2+|u^n|^2 -2)\, u^{n+1/2}, u^{n+1}-u^n \big\rangle 
  = \int_\Omega |u^{n+1}|^4-|u^n|^4 -2\, \big(|u^{n+1}|^2 - |u^n|^2 \big) \dx.
\]
Hence, we get
\[
  \frac\eps2\, \|\nabla u^{n+1}\|^2 - \frac\eps2\, \|\nabla u^n \|^2
  + \frac1{\eps}\, \int_\Omega W(u^{n+1}) - W(u^{n}) \dx
  + \tau\sigma\, \| \nabla w^{n+1/2} \|^2
  = 0, 
\]
which directly implies the claimed energy-dissipation. 
%
\begin{remark}[Besse relaxation]
In order to obtain a scheme which is {\em explicit} in the nonlinearity, one may consider a {\em relaxation} in the sense of~\cite{Bes04} introduced for the nonlinear Schr\"odinger equation. 
For this, we replace~$\frac{1}{2}\, (|u^{n+1}|^2+|u^n|^2)$ in \eqref{eqn:secondOrder:fullCN:b} by a precomputed density $\rho^{n+1/2}$, given by the recursion formula  
%
\begin{align*}
	\rho^{-1/2} 
	\coloneqq |u^0|^2, \qquad
	\rho^{n+1/2} 
	\coloneqq 2\, |u^n|^2 - \rho^{n-1/2}. 
\end{align*}
%
This then results in an \emph{implicit-explicit} variant of the Crank--Nicolson scheme, where each time step only requires the solution of a linear system. One can show that this scheme satisfies the dissipation property for a modified discrete bulk energy. 
Second-order convergence, however, can only be observed for very restrictive parameter regimes in terms of~$\eps$ and~$\sigma$. Because of this, we do not consider this scheme in the following. 
\end{remark}
\begin{remark}
The term $\big(\tfrac{|u^{n+1}|^2+|u^n|^2}{2} - 1\big)\, u^{n+1/2}$ in~\eqref{eqn:secondOrder:fullCN:b} equals the difference quotient of $W(u)=\frac14 (u^2-1)^2$ evaluated at $u^{n}$ and $u^{n+1}$. Therefore, it is a second-order approximation of $W^\prime(u(t^{n+1/2}))$ if the errors for $u^{n}$ and $u^{n+1}$ are of second order. A multiplication by $u^{n+1}-u^{n}$ then gives the difference $W(u^{n+1})-W(u^{n})$. In this way, the here considered approach can be generalized other potentials~$W$. 
\end{remark}
%
%
\subsection{Allen--Cahn type boundary conditions}\label{sec:secondOrder:AllenCahn} 
We now turn to the second-order discretization of system~\eqref{eqn:opEqn:AllenCahn}. For this, we proceed as before, i.e., we consider the Crank--Nicolson discretization where the nonlinear terms are treated as 
\[
  W'(u(t^{n+1/2})) 
  \approx \big( \tfrac{|u^{n+1}|^2+|u^n|^2}{2} - 1 \big)\, u^{n+1/2}, \quad
  W_\Gamma'(p(t^{n+1/2})) 
  \approx \big( \tfrac{|p^{n+1}|^2+|p^n|^2}{2} - 1 \big)\, p^{n+1/2}.
\]
This then leads to the time stepping scheme 
\begin{subequations}
\label{eqn:secondOrder:AllenCahn}
\begin{alignat}{3}
	u^{n+1} + \tau\, \sigma\calK_\Omega w^{n+1/2} 
	&= u^n &&\qquad \text{in } \calVw^\ast, \label{eqn:secondOrder:AllenCahn:a} \\
	\eps\, \calK_\Omega u^{n+1/2} + \eps^{-1} \big(\tfrac{|u^{n+1}|^2+|u^n|^2}{2} - 1\big)\, u^{n+1/2} - \eps\, \calB^\ast \lambda^{n+1}
	&= w^{n+1/2} &&\qquad \text{in } \calVu^\ast \label{eqn:secondOrder:AllenCahn:b}\\
	p^{n+1} + \tau\delta \kappa\, \calK_\Gamma p^{n+1/2} + \tau\delta^{-1} \big(\tfrac{|p^{n+1}|^2+|p^n|^2}{2} - 1\big)\, p^{n+1/2} + \tau\eps\, \lambda^{n+1} 
	&= p^n &&\qquad \text{in } \calVp^\ast, \label{eqn:secondOrder:AllenCahn:c}\\
	\calB u^{n+1} - p^{n+1}  &= 0 &&\qquad \text{in } \calQ^\ast. \label{eqn:secondOrder:AllenCahn:d}
\end{alignat}
\end{subequations}
Note that, in the case of consistent initial data, the constraint~\eqref{eqn:secondOrder:AllenCahn:d} is equivalent to~$\calB u^{n+1/2} - p^{n+1/2} = 0$. For the computation of the initial value for $w$, we refer to Remark~\ref{rem:w0}. We show that the proposed scheme is indeed energy-dissipative. 
\begin{proposition}
\label{prop:secondOrder:AllenCahn}
Assume consistent initial data, i.e., $p^0 = u^0|_\Gamma$. Then, the scheme~\eqref{eqn:secondOrder:AllenCahn} is energy-dissipative, i.e., 
\[
  E(u^{n}) \ge E(u^{n+1})
\]
for all $n\ge 0$. 

\end{proposition}
\begin{proof}
By~\eqref{eqn:secondOrder:AllenCahn:d} we know that~$p^n = u^n|_\Gamma$ for all $n\ge 0$. Now consider the sum of equations~\eqref{eqn:secondOrder:AllenCahn:a}-\eqref{eqn:secondOrder:AllenCahn:c} with test functions~$w^{n+1/2}$, $u^{n+1}-u^n$, and~$\frac1\tau(p^{n+1}-p^n) = \frac 1\tau (u^{n+1}-u^n)|_\Gamma$, respectively. With the equalities obtained in the Neumann-case, this gives 
\begin{align*}
  &\frac\eps2\, \|\nabla u^{n+1}\|^2 - \frac\eps2\, \|\nabla u^n \|^2
  + \frac{1}{\eps}\, \int_\Omega W(u^{n+1}) - W(u^n) \dx + \tau\sigma\, \| \nabla w^{n+1/2} \|^2 \\*
  &+ \frac{\delta\kappa}{2}\, \|\nabla_\Gamma p^{n+1}\|_\Gamma^2 - \frac{\delta\kappa}{2}\, \|\nabla_\Gamma p^n \|_\Gamma^2
  + \frac{1}{\delta}\, \int_\Omega W_\Gamma(p^{n+1}) - W_\Gamma(p^n) \dx
  + \frac1\tau\, \| p^{n+1} - p^n \|^2
  = 0. 
\end{align*}
Hence, we have that 
\[
  E(u^{n+1}) - E(u^{n})
  = - \tau\sigma\, \| \nabla w^{n+1/2} \|^2 - \tau^{-1} \| p^{n+1} - p^n \|^2
  \le 0,
\]
which yields the claim. 	
\end{proof}
\begin{remark}[temporal refinement]
As discussed in Section~\ref{sec:firstOrder:AllenCahn} for the first-order scheme, it is possible to replace equation~\eqref{eqn:secondOrder:AllenCahn:c} by a refined temporal discretization on the boundary. 
This modification maintains the energy-dissipation property of Proposition~\ref{prop:secondOrder:AllenCahn} but, in general, reduces the convergence order to one. 
\end{remark}
\begin{remark}[spatial refinement]
The presented decoupled formulation with the additional variable $p$ on the boundary allows to use different discretization schemes in $\Omega$ and on $\Gamma$. In particular, one may consider a refinement of the spatial mesh used on the boundary, if the solution is, e.g., highly oscillatory. Note that this does not influence the convergence order in time. 
\end{remark}
Next, we turn to dynamic boundary conditions of Cahn--Hilliard type, starting with the model of Liu and Wu.
%
%
\subsection{Boundary conditions of Liu and Wu}\label{sec:secondOrder:LiuWu} 
In this section, we consider the Crank--Nicolson type scheme applied to system~\eqref{eqn:opEqn:LiuWu}. Given consistent initial data $u^0$, $p^0$ and $w^0$, $w_\Gamma^0$ from Remark~\ref{rem:w0}, the resulting scheme reads 
\begin{subequations}
\label{eqn:secondOrder:LiuWu}
\begin{alignat}{3}
	u^{n+1} + \tau\, \sigma\calK_\Omega w^{n+1/2} 
	&= u^n &&\qquad \text{in } \calVw^\ast, \label{eqn:secondOrder:LiuWu:a}\\
	\eps\, \calK_\Omega u^{n+1/2} + \eps^{-1} \big(\tfrac{|u^{n+1}|^2+|u^n|^2}{2} - 1\big)\, u^{n+1/2} - \eps\, \calB^\ast \lambda^{n+1} 
	&= w^{n+1/2} &&\qquad \text{in } \calVu^\ast, \label{eqn:secondOrder:LiuWu:b}\\
	p^{n+1} + \tau\, \calK_\Gamma w_\Gamma^{n+1/2} 
	&= p^n &&\qquad \text{in } \calVp^\ast, \label{eqn:secondOrder:LiuWu:c}\\
	\delta \kappa\, \calK_\Gamma p^{n+1/2} + \delta^{-1} \big(\tfrac{|p^{n+1}|^2+|p^n|^2}{2} - 1\big)\, p^{n+1/2} + \eps\,  \lambda^{n+1} 
	&= w_\Gamma^{n+1/2} &&\qquad \text{in } \calVp^\ast, \label{eqn:secondOrder:LiuWu:d}\\
	\calB u^{n+1} - p^{n+1}
	&= 0 &&\qquad \text{in } \calQ^\ast,
\end{alignat}
\end{subequations}
where the last equation may again be replaced by $\calB u^{n+1/2} - p^{n+1/2} = 0$. 
This scheme satisfies the following dissipation result. 
\begin{proposition}
\label{prop:secondOrder:LiuWu}
Under the assumption of consistent initial data, i.e., $p^0 = u^0|_\Gamma$, the scheme~\eqref{eqn:secondOrder:LiuWu} is energy-dissipative. 
\end{proposition}
\begin{proof}
We proceed similarly as in the previous proof and consider the sum of equations~\eqref{eqn:secondOrder:LiuWu:a}-\eqref{eqn:secondOrder:LiuWu:d} with test functions~$w^{n+1/2}$, $u^{n+1}-u^n$, $w_\Gamma^{n+1/2}$, and~$p^{n+1}-p^n = (u^{n+1}-u^n)|_\Gamma$, respectively. This leads to 
\begin{align*}
	&\frac\eps2\, \|\nabla u^{n+1}\|^2 - \frac\eps2\, \|\nabla u^n \|^2
	+ \frac{1}{\eps}\, \int_\Omega W(u^{n+1}) - W(u^n) \dx + \tau\sigma\, \| \nabla w^{n+1/2} \|^2 \\
	&+ \frac{\delta\kappa}{2}\, \|\nabla_\Gamma p^{n+1}\|_\Gamma^2 - \frac{\delta\kappa}{2}\, \|\nabla_\Gamma p^n \|_\Gamma^2
	+ \frac{1}{\delta}\, \int_\Omega W_\Gamma(p^{n+1}) - W_\Gamma(p^n) \dx + \tau\, \| \nabla_\Gamma w_\Gamma^{n+1/2} \|_\Gamma^2
	= 0, 
\end{align*}
which directly gives $E(u^{n}) \ge E(u^{n+1})$.
\end{proof}
As in the previous model, the PDAE-based formulation allows a refined spatial discretization of the boundary. The possible gain in accuracy is illustrated numerically in Section~\ref{sec:secondOrder:numerics}. 
%
%
\subsection{Boundary conditions of Goldstein, Miranville, and Schimpera}\label{sec:secondOrder:Goldstein} 
The Crank--Nicolson type scheme applied to~\eqref{eqn:opEqn:Goldstein} yields 
\begin{subequations}
\label{eqn:secondOrder:GMS}
\begin{alignat}{3}
	u^{n+1} + \tau\,\sigma \calK_\Omega w^{n+1/2} - \tau\, \calB^\ast \mu^{n+1} 
	&= u^n &&\qquad \text{in } \calVw^\ast, \label{eqn:secondOrder:GMS:a}\\
	\eps\, \calK_\Omega u^{n+1/2} + \eps^{-1} \big(\tfrac{|u^{n+1}|^2+|u^n|^2}{2} - 1\big)\, u^{n+1/2} - \eps\, \calB^\ast \lambda^{n+1} 
	&= w^{n+1/2} &&\qquad \text{in } \calVu^\ast, \label{eqn:secondOrder:GMS:b}\\
	p^{n+1} + \tau\, \calK_\Gamma r^{n+1/2} + \tau \mu^{n+1}
	&= p^n &&\qquad \text{in } \calVp^\ast, \label{eqn:secondOrder:GMS:c}\\
	\delta \kappa\, \calK_\Gamma p^{n+1/2} + \delta^{-1} \big(\tfrac{|p^{n+1}|^2+|p^n|^2}{2} - 1\big)\, p^{n+1/2} + \eps\,  \lambda^{n+1} 
	&= r^{n+1/2} &&\qquad \text{in } \calVp^\ast, \label{eqn:secondOrder:GMS:d}\\
	\calB u^{n+1} - p^{n+1}
	&= 0 &&\qquad \text{in } \calQ^\ast \label{eqn:secondOrder:GMS:e} \\
	\calB w^{n+1} - r^{n+1}
	&= 0 &&\qquad \text{in } \calQ^\ast. \label{eqn:secondOrder:GMS:f}
\end{alignat}
\end{subequations}
Besides the initial data $u^0$ and $p^0$, this scheme also needs values~$w^0$ and $r^0$, cf.~Remark~\ref{rem:w0}. Once more, we discuss the dissipation of energy of the introduced scheme. 
\begin{proposition}
\label{prop:secondOrder:GMS}
Under the assumption of consistent initial data, i.e., $p^0 = u^0|_\Gamma$ and~$r^0 = w^0|_\Gamma$, the scheme~\eqref{eqn:secondOrder:GMS} is energy-dissipative. 
\end{proposition}
\begin{proof}
To prove the dissipation of energy, we consider the sum of equations~\eqref{eqn:secondOrder:GMS:a}-\eqref{eqn:secondOrder:GMS:d} with test functions~$w^{n+1/2}$, $u^{n+1}-u^n$, $r^{n+1/2}$, and~$p^{n+1}-p^n = (u^{n+1}-u^n)|_\Gamma$, respectively. This gives
\begin{align*}
\frac\eps2\, &\|\nabla u^{n+1}\|^2 - \frac\eps2\, \|\nabla u^n \|^2
+ \frac{1}{\eps}\, \int_\Omega W(u^{n+1}) - W(u^n) \dx + \tau\sigma\, \| \nabla w^{n+1/2} \|^2 \\
&+ \frac{\delta\kappa}{2}\, \|\nabla_\Gamma p^{n+1}\|_\Gamma^2 - \frac{\delta\kappa}{2}\, \|\nabla_\Gamma p^n \|_\Gamma^2
+ \frac{1}{\delta}\, \int_\Omega W_\Gamma(p^{n+1}) - W_\Gamma(p^n) \dx + \tau\, \| \nabla_\Gamma r^{n+1/2} \|_\Gamma^2 \\
&+ \eps\, \big\langle \lambda^{n+1}, (p^{n+1}-p^n)-(u^{n+1}-u^n)|_\Gamma \big\rangle_\Gamma
+ \tau\, \big\langle \mu^{n+1}, r^{n+1/2} - w^{n+1/2}|_\Gamma \big\rangle_\Gamma
= 0. 
\end{align*}
By~\eqref{eqn:secondOrder:GMS:e} and the assumed consistency of the initial data, we have $p^n = u^n|_\Gamma$ for all $n\ge0$ such that the $\lambda$-term vanishes. Due to~\eqref{eqn:secondOrder:GMS:f} also the $\mu$-term vanishes, which directly results in~$E(u^{n}) \ge E(u^{n+1})$.
\end{proof}
We close this section on energy-dissipative second-order schemes with a numerical experiment of a two-dimensional model problem with dynamic boundary conditions of Liu--Wu type. 
%
%
\subsection{Numerical experiment}\label{sec:secondOrder:numerics} 
In this final example, we illustrate the positive effect of a spatial refinement of the boundary for the approximation of the variable~$p$. For this, we consider the same system as in Section~\ref{sec:firstOrder:numerics:CH} and fix the spatial mesh used in the bulk (for~$u$). 
Additional refinements of the spatial boundary mesh lead to a decrease of the errors in~$p$ and also slightly in~$u$, cf.~the spatial discretization error in Figure~\ref{fig:error_temp_4_refine_p_L2_H1}. Independent of the spatial meshes, the energy dissipates as predicted by the theory. 

We analyze the convergence order in time for different spatial refinements of~$\Gamma$ (with mesh sizes~$h_\Gamma = 2^{-4},\dots, 2^{-6}$) and a fixed mesh used for~$\Omega$ (with mesh size~$h=2^{-4}$). The errors compared with a reference solution computed on a finer mesh are measured in the $L^2(0,T;H^1(\Gamma))$-norm and illustrated in Figure~\ref{fig:error_temp_4_refine_p_L2_H1}. The plateaus, which are reached by the solid lines, indicate the spatial errors. In contrast, the dashed lines show the pure temporal errors, i.e., the errors based on a reference solution computed on the same spatial mesh. As claimed, one observes that the refinement of the boundary does not influence the convergence in time.
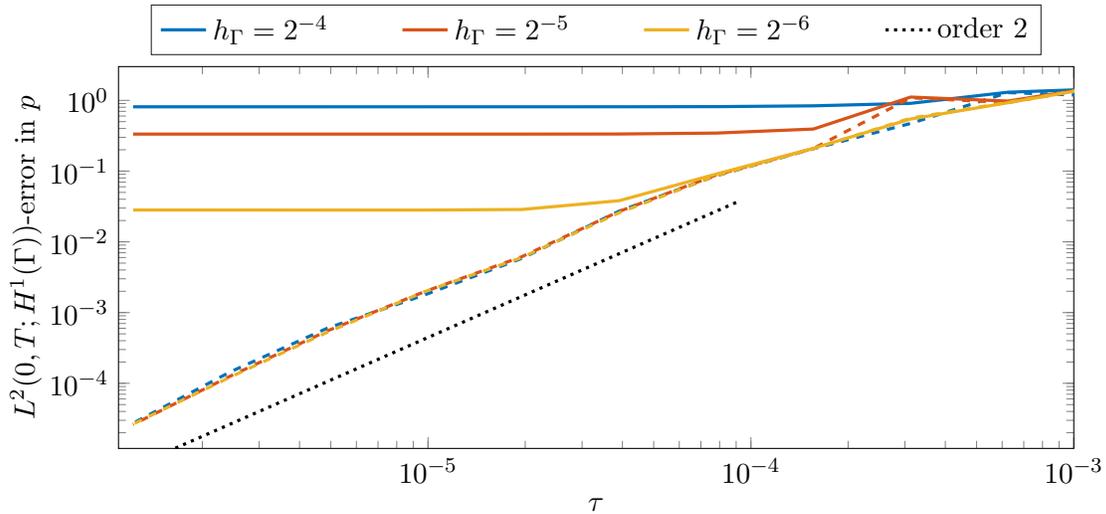
\begin{figure}
\input{pics/convergence_order_ell_L2_H1}
\caption{Errors in the $L^2(0,T;H^1(\Gamma))$-norm over the time step size~$\tau$ for different spatial refinements of the boundary~$\Gamma$. The dashed lines represent the pure temporal errors (without the spatial discretization error).}
\label{fig:error_temp_4_refine_p_L2_H1}
\end{figure}

%% file: pics/convergence_order_ell_L2_H1.tex
%
%
\definecolor{mycolor1}{rgb}{0.00000,0.44700,0.74100}%
\definecolor{mycolor2}{rgb}{0.85000,0.32500,0.09800}%
\definecolor{mycolor3}{rgb}{0.92900,0.69400,0.12500}%
\definecolor{mycolor4}{rgb}{0.49400,0.18400,0.55600}%
\begin{tikzpicture}

\begin{axis}[%
width=5.0in,
height=2.0in,
at={(2.6in,1.27in)},
scale only axis,
xmode=log,
xmin=1.1e-06,
xmax=1e-03,
xlabel = $\tau$,
xminorticks=true,
ymode=log,
ymin=1.2e-05,
ymax=3,
yminorticks=true,
ylabel = {$L^2(0,T;H^1(\Gamma))$-error in $p$},
legend columns = 4,
legend style={legend cell align=left, align=left, at={(0.5,1.03)}, anchor=south, draw=white!15!black}
]

\addplot [color=mycolor1, very thick]
  table[row sep=crcr]{%
0.01	1.18523849239025\\
0.005	1.34734014179076\\
0.0025	1.43050306371158\\
0.00125	1.45827115979738\\
0.000625	1.30219220266616\\
0.0003125	0.907112838517144\\
0.00015625	0.836591526020445\\
7.8125e-05	0.814809691633852\\
3.90625e-05	0.811315306102901\\
1.953125e-05	0.810760258337413\\
9.765625e-06	0.810646866472077\\
4.8828125e-06	0.810617882704352\\
2.44140625e-06	0.810610775070367\\
1.220703125e-06	0.810609010493087\\
};
\addlegendentry{$h_\Gamma = 2^{-4}$ \qquad}

\addplot [color=mycolor1, very thick, dashed, forget plot]
  table[row sep=crcr]{%
0.01	1.3302533312812\\
0.005	1.24626828391913\\
0.0025	1.36461793276171\\
0.00125	1.14941705122985\\
0.000625	1.28641478377724\\
0.0003125	0.471860722110348\\
0.00015625	0.207300778886348\\
7.8125e-05	0.0859330540428367\\
3.90625e-05	0.0271051708593371\\
1.953125e-05	0.00583987767991062\\
9.765625e-06	0.00177829665090621\\
4.8828125e-06	0.0006027435553441\\
2.44140625e-06	0.000145028517289842\\
1.220703125e-06	2.7263027269259e-05\\
6.103515625e-07	0\\
};
\addplot [color=mycolor2, very thick]
  table[row sep=crcr]{%
0.01	1.12538355713033\\
0.005	1.27114958384441\\
0.0025	1.34228300831108\\
0.00125	1.60030591823181\\
0.000625	0.976436919553895\\
0.0003125	1.11315143952546\\
0.00015625	0.393304834771184\\
7.8125e-05	0.344637376340653\\
3.90625e-05	0.334455091073655\\
1.953125e-05	0.333431563544735\\
9.765625e-06	0.333369894193265\\
4.8828125e-06	0.333365522923306\\
2.44140625e-06	0.333367325393214\\
1.220703125e-06	0.333368835366124\\
};
\addlegendentry{$h_\Gamma = 2^{-5}$ \qquad}

\addplot [color=mycolor2, very thick, dashed, forget plot]
  table[row sep=crcr]{%
0.01	1.11136822809007\\
0.005	1.2422711953971\\
0.0025	1.31261846979087\\
0.00125	1.58394967256894\\
0.000625	0.932115840419656\\
0.0003125	1.0859715554182\\
0.00015625	0.209058086819416\\
7.8125e-05	0.088081389252216\\
3.90625e-05	0.0268837506518392\\
1.953125e-05	0.00616322124922795\\
9.765625e-06	0.00197995016218578\\
4.8828125e-06	0.000551929282723509\\
2.44140625e-06	0.000126341192760982\\
1.220703125e-06	2.6163514733303e-05\\
6.103515625e-07	0\\
};
\addplot [color=mycolor3, very thick]
  table[row sep=crcr]{%
0.01	1.16842903664652\\
0.005	1.231547247522\\
0.0025	1.31524300095771\\
0.00125	1.59558036336666\\
0.000625	0.924908876746517\\
0.0003125	0.545024324383904\\
0.00015625	0.210104938296064\\
7.8125e-05	0.0907700355936098\\
3.90625e-05	0.0381107900214114\\
1.953125e-05	0.0286597976040959\\
9.765625e-06	0.0281085505606613\\
4.8828125e-06	0.0280994754544417\\
2.44140625e-06	0.0281372431749851\\
1.220703125e-06	0.0281621569411932\\
};
\addlegendentry{$h_\Gamma = 2^{-6}$ \qquad}

\addplot [color=mycolor3, very thick, dashed, forget plot]
  table[row sep=crcr]{%
0.01	1.18913530340572\\
0.005	1.23852535855475\\
0.0025	1.32253838170965\\
0.00125	1.60989864734733\\
0.000625	0.937072230797501\\
0.0003125	0.560424524569393\\
0.00015625	0.207740106204779\\
7.8125e-05	0.0864464636397689\\
3.90625e-05	0.025725888665608\\
1.953125e-05	0.00600501886388244\\
9.765625e-06	0.00197648935426864\\
4.8828125e-06	0.0005283338092537\\
2.44140625e-06	0.000122587140138945\\
1.220703125e-06	2.64985014106722e-05\\
6.103515625e-07	0\\
};

\addplot [color=black, very thick, dotted]
table[row sep=crcr]{%
	9e-05	3.6e-02\\
	1.5e-06	1e-05\\
};
\addlegendentry{order $2$}
\end{axis}

\end{tikzpicture}%